\documentclass[12pt]{amsart}

\usepackage{amsmath}
\usepackage{amssymb}
\usepackage{amsthm}
\usepackage{pstricks}
\usepackage{pst-node}
\usepackage{pst-plot}
\usepackage{subfigure}
\usepackage{cite}

\numberwithin{equation}{section}

\newtheorem{thm}{Theorem}[section]
\newtheorem{prop}[thm]{Proposition}
\newtheorem{conj}[thm]{Conjecture}
\newtheorem{lem}[thm]{Lemma}
\newtheorem{cor}[thm]{Corollary}

\theoremstyle{definition}
\newtheorem{defin}[thm]{Definition}
\newtheorem{rmk}[thm]{Remark}
\newtheorem{ex}[thm]{Example}

\def\x{{\bf x}}

\def\wt{{\textrm{wt}}}
\def\im{{\textrm{im}}}
\def\opp{{\textrm{opp}}}
\def\field{{\mathbb{C}}}

\begin{document}
\title[Gale-Robinson quivers]{Gale-Robinson quivers: from representations to combinatorial formulas}
\author{Max Glick \and Jerzy Weyman}

\address{Department of Mathematics, The Ohio State University, Columbus, OH 43210, USA}
\address{Department of Mathematics, University of Connecticut, Storrs, CT 06269, USA}
\thanks{J. W. was partially supported by NSF grant DMS-1400740}

\begin{abstract}
We investigate a family of representations of Gale-Robinson quivers that are geared towards providing concrete information about the corresponding cluster algebras.  In this way, we provide a representation theoretic explanation for known combinatorial formulas for the Gale-Robinson sequence and also obtain similar formulas for several other cluster variables.


\end{abstract}

\maketitle


\section{Introduction} \label{secIntro}
The Gale-Robinson recurrence 
\begin{displaymath}
x_{i}x_{i+N} = x_{i+a}x_{i+N-a} + x_{i+c}x_{i+N-c}, \quad i \in \mathbb{Z}
\end{displaymath}
is one of the earliest instances of the Laurent phenomenon \cite{FZ0}: each $x_i$ for $i>N$, when expressed as a function of $x_1, \ldots, x_N$, is given by a Laurent polynomial.  Combinatorial formulas for $x_i$ have been found by D. Speyer \cite{S}, M. Bousquet-M\'elou, J. Propp and J. West \cite{BPW}, and I. Jeong, G. Musiker and S. Zhang \cite{JMZ}.  Remarkably, this simple seeming recurrence finds itself at the intersection of two major areas.
\begin{enumerate}
\item Cluster algebras \cite{FZ1}: There is a quiver for which a certain sequence of mutations generates the $x_i$.  The problem of finding combinatorial formulas for cluster variables in general is very open including the case of variables other than the $x_i$ in Gale-Robinson cluster algebras.  One instance of this latter problem is addressed in \cite{GLVY}.
\item The dimer model on the torus: The Gale-Robinson quiver embeds on a torus and as such gives rise to a dimer model.  R. Eager \cite{E} explores this points of view and identifies the associated toric algebra as the cone over an $L(a,b,c)$ singularity where $b=N-a$.
\end{enumerate}

Our aim is to bring these two takes on Gale-Robinson sequences closer together.  Our main tool is the representation theory of quivers with potential due to H. Derksen, J. Weyman, and A. Zelevinsky \cite{DWZ1,DWZ2}.  Roughly speaking we study a family of representations that are well-structured enough to give combinatorial formulas for the corresponding cluster variables.  These cluster variable end up including the $x_i$ and also several others.  Our approach is most similar to that of Eager and S. Franco \cite{EF} who deal in flavored quivers in place of representations of quivers with potential.

The remainder of this Section introduces the main objects of study including what we call calibrated representations of a Gale-Robinson quiver.  Our main structural result, Theorem \ref{thmMain}, identifies the conditions under which mutation preserves the calibrated property.  Section \ref{secBackground} provides background on cluster algebras and the dimer model.  In Section \ref{secCalibrated} we classify calibrated representations.  In Section \ref{sectau} we turn our focus to mutation of these representations and give a proof of Theorem \ref{thmMain}.  Finally Section \ref{secFPolynomials} identifies the combinatorial formula generated by a calibrated representation (Theorem \ref{thmFM}) and gives several examples.

\subsection{The quiver, the potential, and the lattice}
Fix positive integers $a,b,c,d$ with $a+b=c+d=N$.  Assume $\gcd(a,c,N) = 1$.  The corresponding Gale-Robinson quiver is a quiver $Q$ on vertex set $\{1,2,\ldots, N\}$.  There are eight possible types of arrows $i \to j$ which occur in the following circumstances:
\begin{itemize}
\item $j = i+a$
\item $j = i+b$
\item $j = i-c$
\item $j = i-d$
\item $j = i+a-c$ with $i+a > N$ and $i-c < 1$
\item $j = i+a-d$ with $i+a > N$ and $i-d < 1$
\item $j = i+b-c$ with $i+b > N$ and $i-c < 1$
\item $j = i+b-d$ with $i+b > N$ and $i-d < 1$
\end{itemize}
More precisely, the number of arrows from $i$ to $j$ equals the number of these conditions that are satisfied.  Figure \ref{figQuiver156} gives an example of a Gale-Robinson quiver.

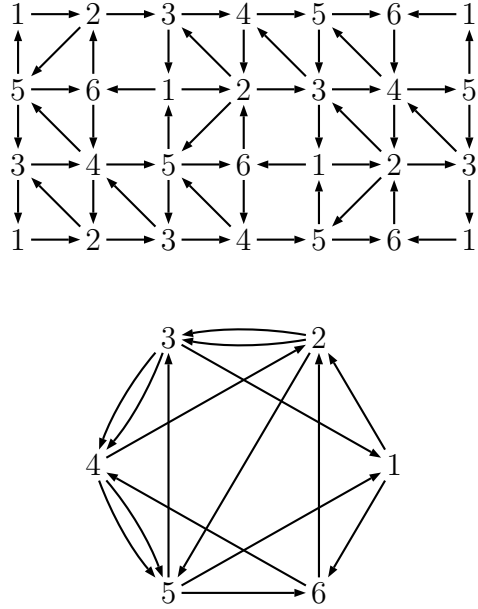
\begin{figure}
\begin{pspicture}(0,1)(7,9)
\rput(0,5){
\rput(1,1){\rnode{v11}{$1$}}
\rput(2,1){\rnode{v21}{$2$}}
\rput(3,1){\rnode{v31}{$3$}}
\rput(4,1){\rnode{v41}{$4$}}
\rput(5,1){\rnode{v51}{$5$}}
\rput(6,1){\rnode{v61}{$6$}}
\rput(7,1){\rnode{v71}{$1$}}
\rput(1,2){\rnode{v12}{$3$}}
\rput(2,2){\rnode{v22}{$4$}}
\rput(3,2){\rnode{v32}{$5$}}
\rput(4,2){\rnode{v42}{$6$}}
\rput(5,2){\rnode{v52}{$1$}}
\rput(6,2){\rnode{v62}{$2$}}
\rput(7,2){\rnode{v72}{$3$}}
\rput(1,3){\rnode{v13}{$5$}}
\rput(2,3){\rnode{v23}{$6$}}
\rput(3,3){\rnode{v33}{$1$}}
\rput(4,3){\rnode{v43}{$2$}}
\rput(5,3){\rnode{v53}{$3$}}
\rput(6,3){\rnode{v63}{$4$}}
\rput(7,3){\rnode{v73}{$5$}}
\rput(1,4){\rnode{v14}{$1$}}
\rput(2,4){\rnode{v24}{$2$}}
\rput(3,4){\rnode{v34}{$3$}}
\rput(4,4){\rnode{v44}{$4$}}
\rput(5,4){\rnode{v54}{$5$}}
\rput(6,4){\rnode{v64}{$6$}}
\rput(7,4){\rnode{v74}{$1$}}
\psset{nodesep=2pt}
\psset{arrowinset=0}
\ncline{->}{v11}{v21}
\ncline{->}{v21}{v31}
\ncline{->}{v31}{v41}
\ncline{->}{v41}{v51}
\ncline{->}{v51}{v61}
\ncline{->}{v71}{v61}
\ncline{->}{v12}{v22}
\ncline{->}{v22}{v32}
\ncline{->}{v32}{v42}
\ncline{->}{v52}{v42}
\ncline{->}{v52}{v62}
\ncline{->}{v62}{v72}
\ncline{->}{v13}{v23}
\ncline{->}{v33}{v23}
\ncline{->}{v33}{v43}
\ncline{->}{v43}{v53}
\ncline{->}{v53}{v63}
\ncline{->}{v63}{v73}
\ncline{->}{v14}{v24}
\ncline{->}{v24}{v34}
\ncline{->}{v34}{v44}
\ncline{->}{v44}{v54}
\ncline{->}{v54}{v64}
\ncline{->}{v74}{v64}
\ncline{->}{v12}{v11}
\ncline{->}{v22}{v21}
\ncline{->}{v32}{v31}
\ncline{->}{v42}{v41}
\ncline{->}{v51}{v52}
\ncline{->}{v61}{v62}
\ncline{->}{v72}{v71}
\ncline{->}{v13}{v12}
\ncline{->}{v23}{v22}
\ncline{->}{v32}{v33}
\ncline{->}{v42}{v43}
\ncline{->}{v53}{v52}
\ncline{->}{v63}{v62}
\ncline{->}{v73}{v72}
\ncline{->}{v13}{v14}
\ncline{->}{v23}{v24}
\ncline{->}{v34}{v33}
\ncline{->}{v44}{v43}
\ncline{->}{v54}{v53}
\ncline{->}{v64}{v63}
\ncline{->}{v73}{v74}
\ncline{->}{v21}{v12}
\ncline{->}{v22}{v13}
\ncline{->}{v24}{v13}
\ncline{->}{v31}{v22}
\ncline{->}{v41}{v32}
\ncline{->}{v43}{v32}
\ncline{->}{v43}{v34}
\ncline{->}{v53}{v44}
\ncline{->}{v62}{v51}
\ncline{->}{v62}{v53}
\ncline{->}{v63}{v54}
\ncline{->}{v72}{v63}
}
\rput(6,3){\rnode{v1}{$1$}}
\rput(5,4.7){\rnode{v2}{$2$}}
\rput(3,4.7){\rnode{v3}{$3$}}
\rput(2,3){\rnode{v4}{$4$}}
\rput(3,1.3){\rnode{v5}{$5$}}
\rput(5,1.3){\rnode{v6}{$6$}}
\psset{nodesep=2pt}
\psset{arrowinset=0}
\ncline{->}{v1}{v2}
\ncarc[arcangle=-10]{->}{v2}{v3}
\ncarc[arcangle=10]{->}{v2}{v3}
\ncarc[arcangle=-10]{->}{v3}{v4}
\ncarc[arcangle=10]{->}{v3}{v4}
\ncarc[arcangle=-10]{->}{v4}{v5}
\ncarc[arcangle=10]{->}{v4}{v5}
\ncline{->}{v5}{v6}
\ncline{->}{v1}{v6}
\ncline{->}{v3}{v1}
\ncline{->}{v4}{v2}
\ncline{->}{v5}{v3}
\ncline{->}{v6}{v4}
\ncline{->}{v5}{v1}
\ncline{->}{v6}{v2}
\ncline{->}{v2}{v5}
\end{pspicture}
\caption{The Gale-Robinson quiver $Q$ (bottom) and its lift $\tilde{Q}$ (top) for $a=1$, $b=5$, $c=2$, $d=4$, $N=6$}
\label{figQuiver156}
\end{figure}

\begin{rmk}
In some cases, the quiver $Q$ has $2$-cycles.  For example, if $(a,b,c) = (1,4,2)$ (so $d=3$ and $N=5$) then there are arrows $2 \to 4$ and $4 \to 2$.  Gale-Robinson quivers are a special case of the family of one-periodic quivers of A. Fordy and R. Marsh \cite{FM}, except that their construction cancels out two cycles when possible.  The above definition with the $2$-cycles intact is due to Eager \cite{E}.
\end{rmk}

The Gale-Robinson quiver can be embedded on the torus.  Its lift $\tilde{Q}$ to the plane has vertex set $\mathbb{Z}^2$.  Each $(x,y) \in \mathbb{Z}^2$ is a lift of the vertex $v \in \{1,2, \ldots, N\}$ for which $v \equiv ax+cy \pmod{N}$.  Each arrow of $\tilde{Q}$ is either a unit vector in one of the four compass directions or a diagonal arrow cutting across a lattice square in one of the mixed directions (e.g. Northeast).  The direction of a lift of a given arrow is determined by its type, as summarized in the first two columns of Table \ref{tabArrowTypes}.  

Thanks to the embedding on a torus, the theory of the dimer model (see e.g. \cite{B,MR}) applies, and we can extract additional structure including
\begin{itemize}
\item a potential $W$ defined as the sum of the boundaries of the counterclockwise faces minus the sum of the boundaries of the clockwise faces,
\item the Jacobian algebra $A$, and
\item a lattice $\Lambda$ together with an associated grading on $A$.
\end{itemize}
The lattice is defined abstractly in terms of the chain complex for $Q$ embedded on the torus.  In the case of Gale-Robinson quivers, there is a homomorphism
\begin{displaymath}
\Lambda \to \mathbb{Z}^4
\end{displaymath}
inducing a grading
\begin{displaymath}
A = \bigoplus_{\lambda \in \mathbb{Z}^4} A_{\lambda}
\end{displaymath}
which is simpler to define.  Each arrow $\alpha$ of $Q$ has a degree $\wt(\alpha) \in \mathbb{Z}^4$ depending on its type as summarized in Table \ref{tabArrowTypes}.  The $\mathbb{Z}^4$ grading captures all the information we need, so we will focus on $\mathbb{Z}^4$ instead of $\Lambda$.  For $\pi$ a path, let $\wt(\pi)$ be the sum of the weights of the arrows of $\pi$.

\begin{table}
\begin{tabular}{r|r|r|r}
$j-i$ & Direction & $(\Delta x, \Delta y)$ & $\wt(\alpha)$ \\
\hline
$a$   & East & $(1,0)$ & $(1,0,0,0)$ \\
$b$   & West & $(-1,0)$ & $(0,1,0,0)$ \\
$-c$  & South & $(0,-1)$ & $(0,0,1,0)$ \\
$-d$  & North & $(0,1)$ & $(0,0,0,1)$ \\
$a-c$ & Southeast & $(1,-1)$ & $(1,0,1,0)$ \\
$a-d$ & Northeast & $(1,1)$ & $(1,0,0,1)$ \\
$b-c$ & Southwest & $(-1,-1)$ & $(0,1,1,0)$ \\
$b-d$ & Northwest & $(-1,1)$ & $(0,1,0,1)$ \\
\end{tabular}
\caption{The types of arrows $\alpha = i \to j$ in a Gale-Robinson quiver}
\label{tabArrowTypes}
\end{table}

\begin{lem} \label{lemWeights}
Let $\pi$ be a path in $Q$ with $\wt(\pi) = \lambda$.  
\begin{itemize}
\item If $\pi$ starts at $i$ and ends at $j$ then $j-i = a\lambda_1 + b\lambda_2 - c\lambda_3 - d\lambda_4$.
\item A lift of $\pi$ to $\tilde{Q}$ starting at vertex $(x,y) \in \mathbb{Z}^2$ will end at vertex $(x+\lambda_1-\lambda_2, y-\lambda_3+\lambda_4)$.
\end{itemize}
\end{lem}

\begin{proof}
The result follows by induction from the single arrow case, which is apparent from Table \ref{tabArrowTypes}.
\end{proof}

\subsection{The representations}
Our main object of study is a family of finite dimensional representations $M$ of the Jacobian algebra $A$ of a Gale-Robinson quiver $Q$.  We are interested in $M$ satisfying several increasingly strict conditions.  First suppose that $M$ is $\mathbb{Z}^4$-graded, that is that it can be decomposed
\begin{displaymath}
M = \bigoplus_{\lambda \in \mathbb{Z}^4} M_{\lambda}
\end{displaymath}
so that if $\pi$ is a path and $x \in M_{\lambda}$ then $\pi x \in M_{\wt(\pi)+\lambda}$.  Say $M$ is \emph{multiplicity free} if
\begin{displaymath}
\dim M_{\lambda} \leq 1
\end{displaymath}
for all $\lambda \in \mathbb{Z}^4$.  

For any graded representation $M$, call $S = \{\lambda \in \mathbb{Z}^4 : \dim M_{\lambda}>0\}$ the \emph{degree set} of $M$.  The vector space $M_v = e_vM$ (where $e_v$ is a lazy path) over each vertex $v \in V$ will also be graded with graded pieces equal to $M_{v,\lambda} = e_vM_{\lambda}$.  Let $S_v$ be the degree set of $M_v$, i.e.
\begin{displaymath}
S_v = \{\lambda \in \mathbb{Z}^4 : \dim M_{v,\lambda}>0\}.
\end{displaymath}
Then $S = S_1 \cup \ldots \cup S_N$, and in the case that $M$ is multiplicity free this union will be disjoint.

\begin{defin} \label{defCalibrated}
Suppose $M$ is a multiplicity free representation of $Q$ with degree set $S$.  Say that $M$ is \emph{calibrated} at level $t \in \mathbb{Z}$ if 
\begin{enumerate}
\item $S_v \subseteq \{\lambda \in \mathbb{Z}^4 : a\lambda_1 + b\lambda_2 - c\lambda_3 - d\lambda_4 = t+v\}$ for all $v=1,2,\ldots, N$, and
\item there is a basis $\{f_{\lambda}: \lambda \in S\}$ of $M$ with each $f_{\lambda} \in M_{\lambda}$ such that 
\begin{equation} \label{eqCalibrated}
\pi f_{\lambda} = \begin{cases}
f_{\wt(\pi)+\lambda}, & \textrm{if } \wt(\pi)+\lambda \in S_v \\
0, & \textrm{otherwise}
\end{cases}
\end{equation}
for each path $\pi$ from $u$ to $v$ and $\lambda \in S_u$.
\end{enumerate}
\end{defin}

Fix $t$ and let $X_t$ be the partially ordered set
\begin{displaymath}
X_t = \{\lambda \in \mathbb{Z}^4 : t+1 \leq a\lambda_1 + b\lambda_2 - c\lambda_3 - d\lambda_4 \leq t+N\}
\end{displaymath}
with the order defined componentwise.  Let $S \subseteq X_t$.  Say $S$ is connected if the corresponding induced subgraph of the Hasse diagram of $X_t$ is connected.  Say $S$ is closed under intervals if for all $x,z \in S$ and $y \in X_t$,
\begin{displaymath}
x \leq y \leq z \Longrightarrow y \in S.
\end{displaymath}

\begin{prop} \label{propMSt}
There exists a calibrated representation of degree set $S \subseteq X_t$ if and only if $S$ is finite and closed under intervals.  Any two calibrated representations with the same degree set $S$ are isomorphic.
\end{prop}

In the above context, let $M(S,t)$ denote the calibrated representation with degree set $S$ at level $t$.

\begin{rmk}
Note that if $S \subseteq X_t$ is closed under intervals and $S'$ is any translation of $S$ in $\mathbb{Z}^4$ then $M(S,t) \cong M(S',t')$ for appropriate $t'$.  In particular, one could always arrange to have $t=0$, but it will be convenient to allow $t$ to vary.
\end{rmk}

\subsection{The main result}
Derksen, Weyman, and Zelevinsky developed a theory of mutation of representations of quivers with potential \cite{DWZ1,DWZ2} that is tailored to the study of cluster algebras.  Non-initial cluster variables correspond in a certain way to representations of $(Q,W)$ related by mutation to simple representations.  We believe that an interesting class of such representations exist among the calibrated representations.  Although small examples can be found by trial and error, there is a systematic procedure to produce potentially infinite families.

Let $(Q,W)$ be a Gale-Robinson quiver with potential as before.  Let $\mu_1$ denote mutation at vertex $1$ and let $\rho$ be the operation that adds one to each vertex label modulo $N$.  On the level of quivers with potential $\mu_1$ is only defined up to right-equivalence, but with appropriate choices $\rho^{-1}(\mu_1(Q,W)) = (Q,W)$.  Given a representation of $(Q,W)$, then, $\rho^{-1} \circ \mu_1$ maps it to another such representation.  Let $\Theta = \rho^{-1} \circ \mu_1$ as a map on the set of representations of $(Q,W)$.  It follows from the general theory that if $M$ corresponds to a cluster variable then so will $\Theta(M)$.  

\begin{prop} \label{proptauGraded}
Suppose $M$ is a graded representation with degree set $S$.  Then $\Theta(M)$ is also graded with a degree set $S'$ satisfying $S'_v = S_{v+1}$ for $v=1,\ldots, N-1$.  Moreover, if
\begin{displaymath}
S_v \subseteq \{\lambda \in \mathbb{Z}^4 : a\lambda_1 + b\lambda_2 - c\lambda_3 - d\lambda_4 = t+v\}
\end{displaymath}
for $v=1,\ldots, N$ then
\begin{displaymath}
S'_N \subseteq \{\lambda \in \mathbb{Z}^4 : a\lambda_1 + b\lambda_2 - c\lambda_3 - d\lambda_4 = t+N+1\}.
\end{displaymath}
\end{prop}

In light of Proposition \ref{proptauGraded}, it is plausible that if $M$ is calibrated at level $t$ then $\Theta(M)$ will be calibrated at level $t+1$.  It turns out an extra condition is needed to ensure that $\Theta(M)$ be multiplicity free.

\begin{defin}
Fix $t \in \mathbb{Z}$.  Say that $S \subseteq X_t$ is \emph{sturdy} if for each $\lambda \in \mathbb{Z}^4$ with $a\lambda_1 + b\lambda_2 - c\lambda_3 - d\lambda_4 = t+1$
\begin{itemize}
\item $\lambda + (1,0,0,0) \in S$ and $\lambda + (0,1,0,0) \in S$ imply $\lambda \in S$, and
\item $\lambda - (0,0,1,0) \in S$ and $\lambda - (0,0,0,1) \in S$ imply $\lambda \in S$.
\end{itemize}
\end{defin}

\begin{thm} \label{thmMain}
Suppose $S \subseteq X_t$ is finite, interval-closed, connected, and sturdy.  Let $M = M(S,t)$ be the corresponding representation of $(Q,W)$, which is calibrated at level $t$.  Then $\Theta(M)$ is calibrated at level $t+1$.  
\end{thm}

\section{Background} \label{secBackground}
\subsection{Cluster algebras and quivers with potential} \label{secDWZ}
\emph{Cluster algebras}, due to S. Fomin and A. Zelevinsky \cite{FZ1}, are certain subalgebras of rational functional fields generated by a recursively defined set of elements known as \emph{cluster variables}.  The starting point is a directed graph $Q$ on $N$ vertices known as a \emph{quiver}.  Assume for the moment that $Q$ does not have any loops or oriented $2$-cycles.  \emph{Quiver mutation} $\mu_k$ of $Q$ at a vertex $k$ produces a new quiver $Q'$ from the following steps
\begin{enumerate}
\item For each length two path $i \to k \to j$ add an arrow $i \to j$.
\item Reverse the orientations of all arrows having $k$ as an endpoint.
\item Cancel in pairs conflicting arrows $i \to j$ and $j \to i$ until no such pairs remain.
\end{enumerate}
A \emph{seed} is a quiver $Q$ together with vertex weights $x_1,\ldots, x_N$.  Quiver mutation at vertex $k$ can be extended to \emph{seed mutation} in which all variables other than $x_k$ are unchanged while $x_k$ is replaced with
\begin{displaymath}
x_k' = \frac{\prod_{i \to k} x_i + \prod_{k \to j} x_j}{x_k}.
\end{displaymath}
The product is over arrows of the initial quiver $Q$ counting multiplicity, so e.g. if there are $3$ arrows from $i$ to $k$ then $x_i^3$ appears in the first term of the numerator.

Starting from an initial seed $(Q, \textbf{x})$, a cluster variable is any variable of any seed reachable by a sequence of mutations.  The corresponding cluster algebra is the subalgebra of $\mathbb{Q}(x_1,\ldots, x_N)$ generated by the set of cluster variables.  The famed Laurent phenomenon \cite{FZ1} asserts that each cluster variable is in fact a Laurent polynomial of $x_1,\ldots, x_N$.  A fundamental problem, that is very difficult in general but that has seen plenty of progress in special cases, is to find combinatorial formulas for these Laurent polynomials.

Derksen, Weyman, and Zelevinsky \cite{DWZ1,DWZ2} developed a theory of quivers with potential in part as a tool to resolve several conjectures \cite{FZ2} concerning the structure of cluster algebras.  A \emph{potential} $W$ is a linear combination of cycles in a quiver $Q$.  We give an overview of the theory of quivers with potential here, although we postpone discussion of several details until the point in the paper at which they are needed.

For the purpose of quivers with potential one allows the quiver to have $2$-cycles, but the mutation rule $\mu_k(Q,W) = (Q',W')$ is only defined for $k$ not part of a $2$-cycle.  The first two steps of mutation are the same as the first two steps of quiver mutation above.  If $\beta\alpha$ is a length two path passing through $k$ (by convention arrows of a path are listed from right to left) let $[\beta\alpha]$ denote the corresponding arrow that got added in the first step.  If $\alpha$ is an arrow having $k$ as one of its endpoint then let $\alpha^*$ denote the reversed arrow which replaced $\alpha$ in the second step.  As for the potential $W$, the following modifications are made for each length two path $\beta \alpha$ passing through $k$:
\begin{itemize}
\item any occurrence of $\beta\alpha$ in any term of $W$ is replaced with the arrow $[\beta\alpha]$,
\item a new term $[\beta\alpha]\alpha^*\beta^*$ is added to $W$ with coefficient $1$.
\end{itemize}
The process that has been carried out so far is called \emph{premutation}.  

The second half of mutation is a process known as \emph{reduction} that modifies both the quiver and potential.  The effect on the quiver is to remove some of the $2$-cycles, so in nice cases the last step of classical quiver mutation is recovered.  The result of reduction (and hence of mutation) is only defined up to a relation on potentials known as \emph{right-equivalence}.  Roughly speaking, two potentials on a quiver are right-equivalent if it is possible to obtain the second from the first by a reversible substitution in which each arrow $\alpha$ from $i$ to $j$ is replaced by a linear combination of paths from $i$ to $j$. 

Throughout we work over the base field $\field$.  The path algebra $\field Q$ of $Q$ is the space of all linear combinations of paths of $Q$ with the product of two paths taken to be their concatenation if it is defined or $0$ otherwise.  Associated to $W$ is its \emph{Jacobian ideal} $J_W$ which is generated by the \emph{cyclic derivatives} of $W$ with respect to the arrows of $Q$.  The \emph{Jacobian algebra} is $A(Q,W) =  \field Q / J_W$.

\begin{rmk}
In \cite{DWZ1} the authors work over the completed path algebra in which infinite linear combinations of paths are allowed.  In examples coming from the dimer model, only finite combinations are ever needed so the above definitions are commonly employed \cite{B,MR}.
\end{rmk}

A \emph{representation} of a quiver with potential $(Q,W)$ is a finite dimensional representation $M$ of $A(Q,W)$.  Letting $e_i \in kQ$ be the lazy (i.e. length zero) path based at a vertex $i$, $M$ is a direct sum of the vector spaces $M_i = e_iM$.  The data of $M$ is captured in these vector spaces together with linear maps $M_{\alpha}:M_i \to M_j$ for each pair of vertices $i,j$ and arrow $\alpha$ from $i$ to $j$.  A \emph{decorated representation} is a representation $M$ together with a choice of $V = \oplus_i V_i$ where each $V_i$ is a vector space.  Let $i$ be a vertex of $Q$.  The \emph{simple positive} (resp. \emph{simple negative}) representation of $(Q,W)$ at $i$ is the one with $V=0$ (resp. $M=0$) and $\dim M_j = \delta_{i,j}$ (resp. $\dim V_j = \delta_{i,j}$).  In both cases all maps $M_{\alpha}$ are necessarily zero.

Let $z$ be a cluster variable in a cluster algebra coming from a quiver with potential.  Derksen, Weyman, and Zelevinsky \cite{DWZ2} associate a decorated representation $M_{(Q,\textbf{x})}(z)$ to each seed $(Q,\textbf{x})$ recursively as follows
\begin{itemize}
\item If $z$ is in the seed, say $z = x_i$, then $M_{(Q,\textbf{x})}(z)$ is the negative simple representation at $i$.
\item If $(Q,\textbf{x})$ and $(Q',\textbf{x}')$ are two seeds related by mutation at $k$ then
\begin{equation} \label{eqMz}
M_{(Q',\textbf{x}')}(z) = \mu_k(M_{(Q,\textbf{x})}(z))
\end{equation}
\end{itemize}
This definition employs a notion of mutation of representations of quivers with potential from \cite{DWZ1}.

Fix a quiver with potential $(Q,W)$.  By \cite{DWZ2}, one can define for each representation $M$ of $(Q,W)$
\begin{itemize}
\item a polynomial $F(M) \in \mathbb{Z}[y_1,\ldots, y_N]$,
\item a vector $g(M) \in \mathbb{Z}^N$, and
\item a nonnegative integer $E(M)$ known as the \emph{$E$-invariant}.
\end{itemize}
The only of these that is not explicitly computable in general is the $F$-polynomial given by
\begin{equation} \label{eqFPoly}
F(M) = \sum_{\textbf{e}\in \left(\mathbb{Z}_{\geq 0}\right)^N} \chi(Gr_{\textbf{e}}(M))y_1^{e_1}\cdots y_N^{e_N}
\end{equation}
where $\chi$ is the Euler characteristic and $Gr_{\textbf{e}}(M)$ is the quiver Grassmannian defined as the variety of subrepresentations of $M$ with dimension vector $\textbf{e}$.  In the case that $M =M_{(Q,\textbf{x})}(z)$ for $z$ a non-initial cluster variable these data satisfy
\begin{itemize}
\item $F(M)$ is the $F$-polynomial of $z$ as defined in \cite{FZ2},
\item $g(M)$ is the $g$-vector of $z$ as defined in \cite{FZ2},
\item $E(M) = 0$.
\end{itemize}
The first two of these imply that the Laurent expression for $z$ can be recovered from $M$ as
\begin{equation} \label{eqxFg}
z = x_1^{g_1}\cdots x_N^{g_N}F(\widehat{y}_1,\ldots, \widehat{y}_N)
\end{equation}
where $(g_1,\ldots, g_N) = g(M)$, $F = F(M)$, and 
\begin{equation} \label{eqyHat}
\widehat{y}_k = \frac{\prod_{k \to j} x_j}{\prod_{i \to k} x_i}
\end{equation}
both products being over arrows in $Q$.

\subsection{The dimer model on the torus}
The dimer model in physics is a theory built around a quiver $Q$ that is embedded on a torus in such a way that for each vertex, the arrows incident to the vertex alternate between incoming and outgoing.  Another way to describe this condition is that the boundary of each face is an oriented cycle.  The dual graph to $Q$ is always bipartite because any pair of neighboring faces of $Q$ includes one that is counterclockwise and one that is clockwise.

There is a potential $W = W(Q)$ associated to any such quiver defined as the sum of the boundaries of the faces of $Q$ with coefficient $\pm 1$ depending on orientation.  We employ the convention that counterclockwise faces have coefficient $+1$ and clockwise faces $-1$.  In general, a mutation $\mu_k$ of $Q$ will produce a quiver no longer on a torus except in the case when $k$ has degree $4$.  Mutation at degree $4$ vertices, also known as toric mutations, gives a way to move from one dimer model to another.

\begin{lem} \label{lemToricMutation}
Let $Q$ be a quiver on a torus, $W = W(Q)$, and let $(Q',W')$ be the result of either 1) premutation at a degree $4$ vertex not part of a $2$-cycle or 2) reduction.  Then $Q'$ is also on a torus and $W'$ is right-equivalent to $W(Q')$.
\end{lem}

\begin{proof}
First consider premutation at a vertex $v$ of degree $4$.  Part (1) of Figure \ref{figtau} shows the local picture of such a quiver with $v=1$, neighboring arrows $n,e,s,w$, and $\pi_1,\ldots, \pi_4$ the paths needed to complete the four faces containing $v$.  Part (2) of the figure makes clear how to embed the result $Q'$ of premutation on the torus.  The potentials $W'$ and $W(Q')$ differ only in the sign of two faces.  Specifically, $[es]s^*e^*$ and $[wn]n^*w^*$ are clockwise so they should have sign $-1$ but actually have sign $+1$ in $W'$.  The right-equivalence $w^* \leftarrow -w^*$, $s^* \leftarrow -s^*$ fixes these two signs without affecting any other terms of the potential.

Now instead consider applying reduction to $(Q,W)$.  The proof is by induction on the number of faces of $Q$ whose boundaries are $2$-cycles.  If there are no such faces then all terms of $W$ have degree at least $3$ so $(Q,W)$ is reduced.  Otherwise, suppose $\beta\alpha$ is a $2$-cycle bounding a counterclockwise face (the clockwise case is similar) of $Q$.  Let $\pi$ and $\pi'$ be the paths so that $\pi \alpha$ and $\pi' \beta$ are the clockwise faces of $Q$ containing $\alpha$ and $\beta$ respectively (see Figure \ref{figReduce}).  Then $W = \beta\alpha - \pi\alpha - \pi'\beta + \ldots$ where the remaining terms do not involve $\alpha$ or $\beta$.  Apply the right-equivalence 
\begin{align*}
\alpha &\leftarrow \alpha + \pi' \\
\beta &\leftarrow \beta + \pi
\end{align*}
The result is
\begin{align*}
\widehat{W} &= (\beta + \pi)(\alpha + \pi') - \pi(\alpha + \pi') - \pi'(\beta + \pi) + \ldots \\
&= \beta\alpha - \pi'\pi + \ldots
\end{align*}
The first step of reduction produces $Q'$ which is $Q$ with the arrows $\alpha, \beta$ removed and $W' = \widehat{W} - \beta\alpha$.  As $\pi'\pi$ is a clockwise face of $Q'$ outside of which nothing has changed, $W' = W(Q')$.  Any remaining $2$-cycle faces can be removed in the same way.
\end{proof}

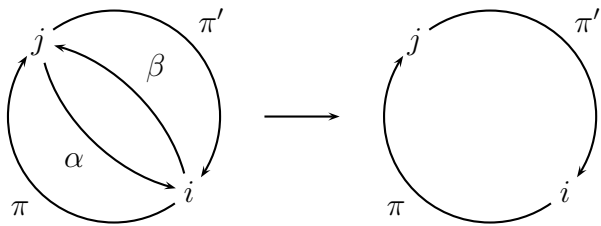
\begin{figure}
\begin{pspicture}(9,4)
\rput(3,1){\rnode{vi}{$i$}}
\rput(1,3){\rnode{vj}{$j$}}
\ncarc[nodesep=3pt,arcangle=30]{<-}{vi}{vj}
\Aput{$\alpha$}
\ncarc[nodesep=3pt,arcangle=30]{<-}{vj}{vi}
\Aput{$\beta$}
\psarc{<-}(2,2){1.41}{-35}{125}
\uput[ur](3,3){$\pi'$}
\psarc{<-}(2,2){1.41}{145}{-55}
\uput[dl](1,1){$\pi$}
\psline{->}(4,2)(5,2)
\rput(5,0){
\rput(3,1){\rnode{vi}{$i$}}
\rput(1,3){\rnode{vj}{$j$}}
\psarc{<-}(2,2){1.41}{-35}{125}
\uput[ur](3,3){$\pi'$}
\psarc{<-}(2,2){1.41}{145}{-55}
\uput[dl](1,1){$\pi$}
}
\end{pspicture}
\caption{A single step of reduction applied to a quiver on a torus}
\label{figReduce}
\end{figure}

\begin{cor} \label{corToricMutation}
Let $Q$ be a quiver on a torus, $W = W(Q)$, and $k$ a vertex of $Q$ of degree $4$ that is not part of a $2$-cycle.  Then
$(Q',W') = \mu_k(Q,W)$ has $Q'$ also on a torus and $W' = W(Q')$ up to right equivalence.
\end{cor}

The cyclic derivative of an arrow $\alpha$ in this setting is $\pi - \pi'$ where $\pi$ and $\pi'$ are paths such that $\pi\alpha$ and $\pi'\alpha$ are the faces of $Q$ containing $\alpha$.  The Jacobian ideal is generated by these cyclic derivatives.  Interest in the dimer model stems from the fact that the Jacobian algebra is always a non-commutative crepant resolution of its center which is a three-dimensional affine toric algebra \cite{B}.

\subsection{Gale-Robinson quivers}
Fix positive integers $a,c,N$ with $a < N$, $c < N$, and $\gcd(a,c,N) = 1$.  The Gale-Robinson sequence $\ldots, x_1, x_2, \ldots$ is defined by the recurrence
\begin{equation} \label{eqGaleRobinson}
x_{i}x_{i+N} = x_{i+a}x_{i+N-a} + x_{i+c}x_{i+N-c}
\end{equation}
for $i \in \mathbb{Z}$.  These sequences exhibit the Laurent phenomenon: each $x_i$ is a Laurent polynomial in the initial data $x_1, \ldots, x_N$.  Combinatorial formulas for Gale-Robinson sequences have been described in various manners \cite{S,BPW,JMZ} and serve as inspiration for the current paper.  

Assuming the two terms on the right hand side of \eqref{eqGaleRobinson} are distinct (i.e. $c \notin \{a,N-A\}$),
there is a quiver with the property that the terms of the sequence are among the corresponding cluster variables.  More precisely, the $x_i$ with $i>N$ are generated by the mutation sequence $1,2,\ldots, N, 1,2 , \ldots$ while the $x_i$ with $i \leq 0$ are generated by the mutation sequence $N, N-1, \ldots, 1, N, N-1 \ldots$.  The quivers in question are precisely the Gale-Robinson quivers defined in the introduction except with any $2$-cycles removed.  

We leave the $2$-cycles intact and allow for $c \in \{a,N-a\}$.  As explained in the introduction, the quiver $Q$ embeds on a torus, so the theory of the dimer model applies.  The corresponding toric three-fold is the affine cone over an $L(a,b,c)$ singularity \cite{E}, where $b = N-a$.

\begin{ex}
We can assume without loss of generality that $a \leq c \leq N/2$.  Some small examples include
\begin{itemize}
\item $(a,b,c) = (1,1,1)$: the \emph{conifold} quiver
\item $(a,b,c) = (1,2,1)$: the \emph{suspended pinch point} quiver
\item $(a,b,c) = (1,3,2)$: the \emph{Somos-$4$} (a.k.a. del Pezzo-1) quiver
\item $(a,b,c) = (1,4,2)$: the \emph{Somos-$5$} quiver
\end{itemize}
\end{ex}

Let $Q$ be a Gale-Robinson quiver and $W = W(Q)$.

\begin{prop} \label{propAGraded}
The Jacobian algebra $A = A(Q,W)$ has a $\mathbb{Z}^4$ grading so that each individual arrow lies in a graded piece as described in the last column of Table \ref{tabArrowTypes}.
\end{prop}

\begin{proof}
The weights from the table define a $\mathbb{Z}^4$ grading on $\field Q$ where the weight of a path equals the sum of the weights of its arrows.  As $A = \field Q / J_W$, we must show that the Jacobian ideal is homogeneous.  A typical generator is $\pi - \pi'$ where $\pi\alpha$ and $\pi'\alpha$ are faces of $Q$, so we need to show $\wt(\pi) = \wt(\pi')$ in this case.  This condition is equivalent to $\wt(\pi\alpha) = \wt(\pi'\alpha)$ so it suffices to show that the boundary of each face of $Q$ has the same weight.  Each face is either a quadrilateral with one arrow in each main compass direction or a triangle with two perpendicular arrows in main compass directions from which the third direction is determined (e.g. North followed by East must be closed up by Southwest).  One can check using Table \ref{tabArrowTypes} that in all cases the total weight of the boundary is $(1,1,1,1)$.
\end{proof}

\begin{rmk}
Any Jacobian algebra coming from the dimer model is graded by a lattice $\Lambda$, see \cite{B} and \cite{MR}.  In the case of Gale-Robinson quivers the resulting grading is related to the $\mathbb{Z}^4$-grading just defined.  The $\mathbb{Z}^4$-grading itself can be obtained more conceptually via a description of the Jacobian algebra in terms of the ring of invariants of a certain group action on polynomials in four variables \cite{E}.
\end{rmk}

Now fix a representation $M$ of $(Q,W)$ that is $\mathbb{Z}^4$-graded.  Viewing $M$ as a vector space, the torus $T = (\mathbb{C}^*)^4$ acts in a natural way.  Specifically, a given $z \in T$ acts on each graded piece $M_{\lambda}$ of $M$ by multiplication by the scalar $z_1^{\lambda_1}z_2^{\lambda_2}z_3^{\lambda_3}z_4^{\lambda_4}$.  We explain how this action extends to an action on the quiver Grassmannians of $M$.  There is a very similar discussion of a torus action on Hilbert schemes in Section 2 of \cite{MR}, although we are not sure if there is a direct connection.  

\begin{prop}
Let $P \subseteq M$ be a subrepresentation and $z \in T$.  Then $zP \subseteq M$ is also a subrepresentation.
\end{prop}

\begin{proof}
The content of this Proposition is that if $p \in P$ and $\alpha$ is an arrow of $Q$ then $M_{\alpha}(zp) = zp'$ for some $p' \in P$.  We can write $p = \sum_{\lambda} x_{\lambda}$ with each $x_{\lambda} \in M_{\lambda}$.  Then
\begin{align*}
M_{\alpha}(zp) &= M_{\alpha}\left(\sum_{\lambda} z^{\lambda}x_{\lambda}\right) \\
&= \sum_{\lambda} z^{\lambda}M_{\alpha}(x_{\lambda}) \\
&= z^{-\mu}\sum_{\lambda} z^{(\lambda+\mu)}M_{\alpha}(x_{\lambda})
\end{align*}
where $\mu = \wt(\alpha)$.  Now $M_{\alpha}(x_{\lambda})$ lies in degree $\lambda + \mu$ so we obtain $M_{\alpha}(zp) = zp'$ where
\begin{displaymath}
p' = z^{-\mu}\sum_{\lambda} M_{\alpha}(x_{\lambda})
\end{displaymath}
\end{proof}

It is easy to see that $zP$ and $P$ above have the same dimension vectors, so we obtain an action of the torus on $Gr_e(M)$ for any given $e \in \mathbb{Z}_{\geq 0}^N$.  As explained in \cite{MR}, we can restrict to torus fixed points when calculating the Euler characteristic.  As such we can rewrite \eqref{eqFPoly} as
\begin{equation} \label{eqFPolyFix}
F(M) = \sum_{\textbf{e}} \chi(Gr_{\textbf{e}}(M)^T)y_1^{e_1}\cdots y_N^{e_N}.
\end{equation}
It is typically much easier to calculate the Euler characteristic of such a fixed space compared to that of the full quiver Grassmannian.  The following Proposition identifies what the fixed points are and should be compared to \cite[Theorem 2.4]{MR}.

\begin{prop} \label{propTFixed}
A subrepresentation $P$ of $M$ is fixed by $T$ if and only if $P$ is itself $\mathbb{Z}^4$-graded.
\end{prop}

\begin{proof}
Each $p \in P$ can be written uniquely as $p = \sum_{\lambda} x_{\lambda}$ with the $x_{\lambda} \in M_{\lambda}$.  Suppose for the sake of contradiction that $P$ is fixed by the torus but not graded.  Then there exists $p$ as above such that $x_{\lambda} \notin P$ for at least one $\lambda \in \mathbb{Z}^4$.  Choose $p$ so that the number of such $\lambda$ is as small as possible.  If there were just one such $\lambda$ then
\begin{displaymath}
x_{\lambda} = p - \sum_{\lambda' \neq \lambda} x_{\lambda'} \in P
\end{displaymath}
a contradiction.  Otherwise there are $\lambda \neq \mu$ so that $x_{\lambda} \notin P$ and $x_{\mu} \notin P$.  Pick some $z \in T$ so that $z^{\lambda} \neq z^{\mu}$.  Then a linear combination $p'$ of $p$ and $zp$ can be chosen in such a way that the coefficient of $x_{\lambda}$ is $0$ but the coefficient of $x_{\mu}$ is not.  We have $p' \in P$ (since $P$ is fixed by $T$) and not all graded pieces of $p'$ are in $P$ (since $x_{\mu} \notin P$).  However $p'$ has fewer such pieces than $p$ since $x_{\lambda}$ does not appear, contradicting the minimality assumption.
\end{proof}

\section{Classification of calibrated representations} \label{secCalibrated}
We turn now to the classification of calibrated representations which was already stated in Proposition \ref{propMSt}.  In short, a calibrated representation is determined up to isomorphism by its degree set which must be an interval-closed subset of the poset $X_t$.  We will need the following result which together with the first part of Lemma \ref{lemWeights} characterizes weights of paths.

\begin{prop} \label{proplambda2pi}
Let $u,v \in \{1,2,\ldots, N\}$ and suppose $\lambda \in \mathbb{Z}_{\geq 0}^4$ with
\begin{displaymath}
a\lambda_1 + b\lambda_2 - c\lambda_3 - d\lambda_4 = v-u.
\end{displaymath}
Then there is a path $\pi$ in $Q$ from $u$ to $v$ with weigh $\lambda$.
\end{prop}

\begin{proof}
The proof is by induction on $\lambda_1 + \ldots + \lambda_4$.  If $\lambda = (0,0,0,0)$ then take the length $0$ path at $u=v$.  Now suppose at least one $\lambda_i > 0$.  For concreteness we look at the case $\lambda_1 > 0$.  If $u \leq N-a$ then there is an East arrow from $u$ to $u+a$ of weight $(1,0,0,0)$.  The induction hypothesis then supplies a path from $u+a$ to $v$ of weight $\lambda - (1,0,0,0)$.  If $u > N-a$ then we have
\begin{displaymath}
c\lambda_3 + d\lambda_4 = u-v + a\lambda_1 + b\lambda_2 > (N-a) + (-N) + a(1) = 0
\end{displaymath}
so one of $\lambda_3$ or $\lambda_4$ is positive, say $\lambda_3$.  The first arrow of the path can then be taken to be a South arrow from $u$ to $u-c$ or a Southeast arrow from $u$ to $u+a-c$ depending on if $u>c$.
\end{proof}

\begin{proof}[Proof of Proposition \ref{propMSt}]
Uniqueness is straightforward.  We only deal in finite dimensional representations which explains why $S$ must be finite.  Given finite $S$, we can define $M$ to be an abstract vector space with basis $\{f_\lambda : \lambda \in S\}$.  The first question is whether  \eqref{eqCalibrated} extended by linearity gives an action of the path algebra of $Q$ on $M$.  The condition to check is
\begin{displaymath}
(\pi'\pi)f_{\lambda} = \pi'(\pi f_{\lambda})
\end{displaymath}
for a pair of paths $\pi$ and $\pi'$.  Both sides are $0$ unless $\lambda \in S_u$, $\pi$ goes from $u$ to $v$, and $\pi'$ goes from $v$ to $w$.  Both sides are also $0$ if $\lambda + \wt(\pi) + \wt(\pi') \notin S_w$.  The only way for the condition to fail is if $\lambda + \wt(\pi) \notin S_v \subseteq S$ but $\lambda + \wt(\pi) + \wt(\pi') \in S_w \subseteq S$.  If $S$ is closed under intervals then this is impossible and we have an action of $\field Q$.  If not then let $\lambda \leq \mu \leq \nu$ with $\lambda, \nu \in S$ and $\mu \notin S$.  Then Proposition \ref{proplambda2pi} applied to $\mu - \lambda$ and $\nu - \mu$ gives paths $\pi$ and $\pi'$ violating the needed condition.

It remains to show that if $S$ is closed under intervals then the action of $\field Q$ on $M$ descends to an action of the Jacobian algebra.  As explained in the proof of Proposition \ref{propAGraded}, if $\pi'-\pi$ is one of the standard generators of the Jacobian ideal then $wt(\pi)=wt(\pi')$.  Therefore $\pi$ and $\pi'$ act identically on $M$ as desired.
\end{proof}

\section{The action of $\Theta$} \label{sectau}
\subsection{Definition of $\Theta$}
We now fill in the details for the definition of $\Theta$.  Let $Q$ be a Gale-Robinson quiver corresponding to the parameters $a,b,c,d$ with $N = a+b = c+d$.  The vertex $1$ always has degree $4$ with outgoing arrows to $1+a$ and $1+b$ and incoming arrows from $1+c$ and $1+d$.  We place the assumption $\{a,b\} \neq \{c,d\}$ in this Section so that $1$ is not part of a $2$-cycle.  Hence by Corollary \ref{corToricMutation} we can perform the mutation $\mu_1$.  In total, $\Theta$ can be broken down into five steps which are pictured in Figure \ref{figtau}.

\begin{figure}
\begin{pspicture}(12,18)
\psset{labelsep=3pt}
\rput(0,12){
\rput(.5,5.5){$(1)$}
\rput(3,3){\rnode{v}{$1$}}
\rput(5,3){\rnode{va}{$1+a$}}
\rput(1,3){\rnode{vb}{$1+b$}}
\rput(3,5){\rnode{vc}{$1+c$}}
\rput(3,1){\rnode{vd}{$1+d$}}
\ncline[nodesep=3pt]{->}{v}{va}
\Aput{$e$}
\ncline[nodesep=3pt]{->}{v}{vb}
\Bput{$w$}
\ncline[nodesep=3pt]{->}{vc}{v}
\Aput{$s$}
\ncline[nodesep=3pt]{->}{vd}{v}
\Bput{$n$}
\ncarc[nodesep=3pt,arcangle=-45]{->}{va}{vc}
\Bput{$\pi_1$}
\ncarc[nodesep=3pt,arcangle=45]{->}{va}{vd}
\Aput{$\pi_2$}
\ncarc[nodesep=3pt,arcangle=45]{->}{vb}{vc}
\Aput{$\pi_3$}
\ncarc[nodesep=3pt,arcangle=-45]{->}{vb}{vd}
\Bput{$\pi_4$}
\rput(2,2){$+$}
\rput(4,2){$-$}
\rput(2,4){$-$}
\rput(4,4){$+$}
}
\rput(7,12){
\rput(.5,5.5){$(2)$}
\rput(3,3){\rnode{v}{$1$}}
\rput(5,3){\rnode{va}{$1+a$}}
\rput(1,3){\rnode{vb}{$1+b$}}
\rput(3,5){\rnode{vc}{$1+c$}}
\rput(3,1){\rnode{vd}{$1+d$}}
\ncline[nodesep=3pt]{->}{va}{v}
\Bput{$e^*$}
\ncline[nodesep=3pt]{->}{vb}{v}
\Aput{$w^*$}
\ncline[nodesep=3pt]{->}{v}{vc}
\Bput{$s^*$}
\ncline[nodesep=3pt]{->}{v}{vd}
\Aput{$n^*$}
\ncline[nodesep=3pt]{->}{vc}{va}
\Aput{$[es]$}
\ncline[nodesep=3pt]{->}{vc}{vb}
\Bput{$[ws]$}
\ncline[nodesep=3pt]{->}{vd}{va}
\Bput{$[en]$}
\ncline[nodesep=3pt]{->}{vd}{vb}
\Aput{$[wn]$}
\ncarc[nodesep=3pt,arcangle=-90]{->}{va}{vc}
\Bput{$\pi_1$}
\ncarc[nodesep=3pt,arcangle=90]{->}{va}{vd}
\Aput{$\pi_2$}
\ncarc[nodesep=3pt,arcangle=90]{->}{vb}{vc}
\Aput{$\pi_3$}
\ncarc[nodesep=3pt,arcangle=-90]{->}{vb}{vd}
\Bput{$\pi_4$}
\rput(2.3,2.3){$+$}
\rput(3.7,2.3){$+$}
\rput(2.3,3.7){$+$}
\rput(3.7,3.7){$+$}
\rput(1,2.2){$+$}
\rput(5,2.2){$-$}
\rput(1,3.8){$-$}
\rput(5,3.8){$+$}
}
\rput(0,6){
\rput(.5,5.5){$(3)$}
\rput(3,3){\rnode{v}{$1$}}
\rput(5,3){\rnode{va}{$1+a$}}
\rput(1,3){\rnode{vb}{$1+b$}}
\rput(3,5){\rnode{vc}{$1+c$}}
\rput(3,1){\rnode{vd}{$1+d$}}
\ncline[nodesep=3pt]{->}{va}{v}
\Bput{$e^*$}
\ncline[nodesep=3pt]{->}{vb}{v}
\Aput{$-w^*$}
\ncline[nodesep=3pt]{->}{v}{vc}
\Bput{$-s^*$}
\ncline[nodesep=3pt]{->}{v}{vd}
\Aput{$n^*$}
\ncline[nodesep=3pt]{->}{vc}{va}
\Aput{$[es]$}
\ncline[nodesep=3pt]{->}{vc}{vb}
\Bput{$[ws]$}
\ncline[nodesep=3pt]{->}{vd}{va}
\Bput{$[en]$}
\ncline[nodesep=3pt]{->}{vd}{vb}
\Aput{$[wn]$}
\ncarc[nodesep=3pt,arcangle=-90]{->}{va}{vc}
\Bput{$\pi_1$}
\ncarc[nodesep=3pt,arcangle=90]{->}{va}{vd}
\Aput{$\pi_2$}
\ncarc[nodesep=3pt,arcangle=90]{->}{vb}{vc}
\Aput{$\pi_3$}
\ncarc[nodesep=3pt,arcangle=-90]{->}{vb}{vd}
\Bput{$\pi_4$}
\rput(2.3,2.3){$-$}
\rput(3.7,2.3){$+$}
\rput(2.3,3.7){$+$}
\rput(3.7,3.7){$-$}
\rput(1,2.2){$+$}
\rput(5,2.2){$-$}
\rput(1,3.8){$-$}
\rput(5,3.8){$+$}
}
\rput(7,6){
\rput(.5,5.5){$(4)$}
\rput(3,3){\rnode{v}{$1$}}
\rput(5,3){\rnode{va}{$1+a$}}
\rput(1,3){\rnode{vb}{$1+b$}}
\rput(3,5){\rnode{vc}{$1+c$}}
\rput(3,1){\rnode{vd}{$1+d$}}
\ncline[nodesep=3pt]{->}{va}{v}
\Bput{$e^*$}
\ncline[nodesep=3pt]{->}{vb}{v}
\Aput{$-w^*$}
\ncline[nodesep=3pt]{->}{v}{vc}
\Bput{$-s^*$}
\ncline[nodesep=3pt]{->}{v}{vd}
\Aput{$n^*$}
\ncline[nodesep=3pt]{->}{vc}{vb}
\Bput{$[ws]$}
\ncline[nodesep=3pt]{->}{vd}{vb}
\Aput{$[wn]$}
\ncarc[nodesep=3pt,arcangle=45]{->}{vc}{va}
\Aput{$\pi_1'$}
\ncarc[nodesep=3pt,arcangle=-45]{->}{vd}{va}
\Bput{$\pi_2'$}
\ncarc[nodesep=3pt,arcangle=90]{->}{vb}{vc}
\Aput{$\pi_3$}
\ncarc[nodesep=3pt,arcangle=-90]{->}{vb}{vd}
\Bput{$\pi_4$}
\rput(2.3,2.3){$-$}
\rput(3.7,2.3){$+$}
\rput(2.3,3.7){$+$}
\rput(3.7,3.7){$-$}
\rput(1,2.2){$+$}
\rput(1,3.8){$-$}
}
\rput(0,0){
\rput(.5,5.5){$(5)$}
\rput(3,3){\rnode{v}{$N$}}
\rput(5,3){\rnode{va}{$a$}}
\rput(1,3){\rnode{vb}{$b$}}
\rput(3,5){\rnode{vc}{$c$}}
\rput(3,1){\rnode{vd}{$d$}}
\ncline[nodesep=3pt]{->}{va}{v}
\Bput{$e^*$}
\ncline[nodesep=3pt]{->}{vb}{v}
\Aput{$-w^*$}
\ncline[nodesep=3pt]{->}{v}{vc}
\Bput{$-s^*$}
\ncline[nodesep=3pt]{->}{v}{vd}
\Aput{$n^*$}

\ncarc[nodesep=3pt,arcangle=45]{->}{vc}{va}
\Aput{$\pi_1'$}
\ncarc[nodesep=3pt,arcangle=-45]{->}{vd}{va}
\Bput{$\pi_2'$}
\ncarc[nodesep=3pt,arcangle=-45]{->}{vc}{vb}
\Bput{$[ws]$}
\ncarc[nodesep=3pt,arcangle=45]{->}{vd}{vb}
\Aput{$[wn]$}
\rput(2.3,2.3){$-$}
\rput(3.7,2.3){$+$}
\rput(2.3,3.7){$+$}
\rput(3.7,3.7){$-$}
}
\psset{labelsep=5pt}
\end{pspicture}
\caption{A schematic of the operation $\Theta$.}
\label{figtau}
\end{figure}
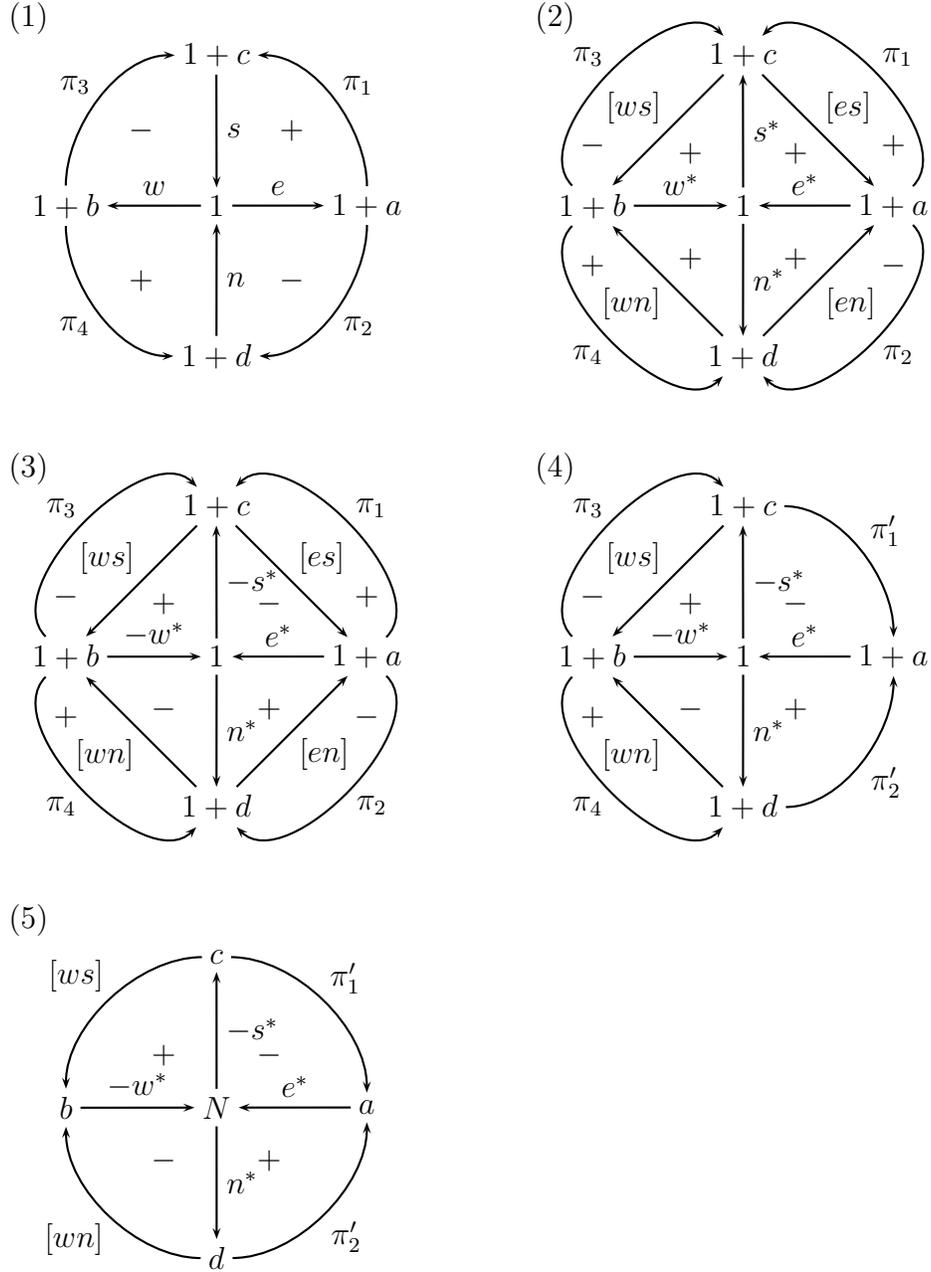

\begin{enumerate}
\item Begin with a triple $(Q,W,M)$ where $Q$ is the Gale-Robinson quiver, $W=W(Q)$, and $M$ is a representation of $(Q,W)$.  Let $e, w, s, n$ be the arrows incident to vertex $1$ as in Figure \ref{figtau} and use the same names for the corresponding maps of $M$.  For instance, $e$ is a linear from $e: M_1 \to M_{1+a}$.  Let $\pi_1,\ldots, \pi_4$ be the paths needed to complete the four faces touching vertex $1$, as in Figure \ref{figtau}.  Each $\pi_i$ is comprised of either a single diagonal arrow or  a pair of arrows, the first vertical and the second horizontal.  The label of a face indicates the sign with which its boundary cycle contributes to $W$.
\item Apply premutation at vertex $1$.  The four arrows incident to $1$ are replaced by their reversals, and four more arrows are added.  Note that it is no longer the case that the sign of each face in the potential is determined by its orientation.
\item Apply the right-equivalence that negates arrows $w^*$ and $s^*$ and fixes all other arrows.  This resolves the sign issue in the potential created during the previous step.
\item For each $\pi_i$ that is a single arrow, reduce out the $2$-cycle that it forms with one of the new arrows.  This reduction step can be performed in a way so that no arrows other than those in the $2$-cycle are modified and also the new face that is created appears in the potential with coefficient $\pm 1$ as dictated by its orientation (see the proof of Lemma \ref{lemToricMutation}).  For each $\pi_i$ that was eliminated, let $\pi_i'$ be the path in $Q$ so that $\pi_i'\pi_i$ is the face of $Q$ containing $\pi_i$ that does not include vertex $1$.  Such $\pi_i'$ always consists of a horizontal followed by a vertical arrow.  Figure \ref{figtau} depicts the case in which $\pi_1$ and $\pi_2$ are single arrows while $\pi_3$ and $\pi_4$ are not.  In fact, this situation can always be arranged by assuming $a<c\leq N/2$.
\item Subtract $1$ from each vertex label modulo $N$ obtaining $(Q,W,\Theta(M))$, the original quiver and potential with a new representation.  Note in the case of Figure \ref{figtau}, the paths $\pi_3$ and $\pi_4$ remain at the end but they are outside the area depicted in the last frame.
\end{enumerate}

We generally assume above that $M$ is an ordinary representation (or equivalently, a decorated representation with $V=0$).  Even in this case the outcome of mutation might have nonzero decoration.  One could allow for this possibility, but we simply declare $\Theta$ to be undefined if the result would be decorated.  In the case where $M$ is indecomposable, $\Theta$ will only be undefined for $M$ equal to the simple representation at $1$.

We review the definition of premutation of a representation \cite{DWZ1} restricted to the situation in step (2) where it is carried out (and ignoring the possibility that the outcome could be decorated as per the previous paragraph).  Recall the initial representation $M$ can be described by a vector space $M_v$ associated to each vertex of $Q$ together with a linear map associated to each arrow.  The result $M'$ of premutation at $1$ has $M_v' = M_v$ for all $v \neq 1$.  The vector space $M_1'$ is expressed in terms of a triangle of linear maps
\begin{align*}
\alpha &: M_{1+c} \oplus M_{1+d} \longrightarrow M_1 \\
\beta &: M_1 \longrightarrow M_{1+a} \oplus M_{1+b}  \\
\gamma &: M_{1+a} \oplus M_{1+b} \longrightarrow M_{1+c} \oplus M_{1+d}
\end{align*}
as
\begin{displaymath}
M_1' = \frac{\ker \gamma}{im \beta} \oplus im \gamma \oplus \frac{\ker \alpha}{im \gamma}.
\end{displaymath}
The maps $\alpha$ and $\beta$ are built respectively out of the incoming and outgoing arrows to vertex $1$.  In matrix form
\begin{align*}
\alpha &= \left[ \begin{array}{cc} s & n \\ \end{array} \right] \\
\beta &= \left[ \begin{array}{c} e \\ w \\ \end{array} \right] 
\end{align*}
Meanwhile, $\gamma$ is built out of ``second derivatives'' of the potential 
\begin{displaymath}
W = \pi_1es - \pi_2en - \pi_3ws + \pi_4wn + \ldots
\end{displaymath}
with respect to length two paths passing through $1$.  The result is
\begin{equation} \label{eqgamma}
\gamma = \left[\begin{array}{cc} \pi_1 & -\pi_3 \\ -\pi_2 & \pi_4 \\ \end{array} \right].
\end{equation}

The linear maps associated to the arrows not touching vertex $1$ are unaffected by premutation.  The four new diagonal arrows each get assigned a composition of two linear maps in the manner suggested by the naming convention.  For instance $[es] = e \circ s$.  It remains to define $e^*$, $w^*$, $s^*$, and $n^*$ as linear maps.  These maps depend on a choice of splitting data, a pair of linear maps
\begin{align*}
\rho &: M_{1+a} \oplus M_{1+b} \to \ker \gamma \\
\sigma &: \ker \alpha / \im \gamma \to \ker \alpha 
\end{align*}
such that $\rho$ restricts to the identity on $\ker \gamma$ and $\sigma$ sends each coset to one of its representatives.  Given such a choice
\begin{align}
\left[\begin{array}{cc} e^* & w^* \end{array}\right]
&= \left[\begin{array}{c} - \pi \rho \\ -\gamma \\ 0 \\ \end{array} \right] \label{eqalphaStar}\\
\left[\begin{array}{c} s^* \\ n^*\end{array}\right] 
&= \left[\begin{array}{ccc} 0 & \iota & \iota'\sigma \end{array} \right] \label{eqbetaStar}
\end{align}
where 
\begin{displaymath}
\pi : \ker \gamma \to \ker \gamma / \im \beta
\end{displaymath}
is projection and
\begin{align*}
\iota &: \im \gamma \to M_{1+c} \oplus M_{1+d} \\
\iota' &: \ker \alpha \to M_{1+c} \oplus M_{1+d}
\end{align*}
are inclusions.

\subsection{Compatibility of $\Theta$ with grading}
The next order of business is to prove Proposition \ref{proptauGraded}, which asserts that $\Theta$ maps graded representations to graded representations.  The proof boils down to tracing the grading through all the auxiliary maps of the construction.  We keep all of the notation from the previous subsection.
\begin{lem} \label{lemNewGrading}
Let $M$ be a graded representation of $(Q,W)$ and let $(Q',W',M')$ be the result of premutation at vertex $1$.  There is a $\mathbb{Z}^4$-grading on $M_1'$ which together with the gradings already in place for $M_v'=M_v$, $v \neq 2$ has the property that all arrows of $Q'$ act by a constant shift of degree.  The shift associated to $e^*$ is $(0,1,0,0)$, i.e. $e^*(M'_{1+a,\lambda}) \subseteq M_{1,\lambda+(0,1,0,0)}'$.  The shifts associated to $w^*, s^*, n^*$ are $(1,0,0,0)$, $(0,0,0,1)$, and $(0,0,1,0)$ respectively.  
\end{lem}

\begin{proof}
The only difficulty is to define the grading on $M_1'$ and to understand the four arrows of $Q'$ incident to $1$.  The maps $\alpha$, $\beta$, and $\gamma$ can be broken up into blocks
\begin{align*}
\alpha_{\lambda} : M_{1+c,\lambda+(0,0,0,1)} \oplus M_{1+d,\lambda+(0,0,1,0)} &\longrightarrow M_{1,\lambda+(0,0,1,1)} \\
\beta_{\lambda} : M_{1,\lambda-(1,1,0,0)} &\longrightarrow M_{1+a,\lambda-(0,1,0,0)} \oplus M_{1+b,\lambda-(1,0,0,0)}  \\
\gamma_{\lambda} : M_{1+a,\lambda-(0,1,0,0)} \oplus M_{1+b,\lambda-(1,0,0,0)} &\longrightarrow M_{1+c,\lambda+(0,0,0,1)} \oplus M_{1+d,\lambda+(0,0,1,0)}
\end{align*}
For example, the two components of $\alpha$ are $s$ and $n$ which have weight $(0,0,1,0)$ and $(0,0,0,1)$ respectively.  It follows that
\begin{align*}
s(M_{1+c,\lambda+(0,0,0,1)}) &\subseteq M_{1,\lambda+(0,0,1,1)} \\
n(M_{1+d,\lambda+(0,0,1,0)}) &\subseteq M_{1,\lambda+(0,0,1,1)}
\end{align*}
from which we can build $\alpha_{\lambda}$.  

The blocks line up so that the compositions $\gamma_{\lambda} \circ \beta_{\lambda}$ and $\alpha_{\lambda} \circ \gamma_{\lambda}$ exist.  Letting
\begin{displaymath}
M_{1,\lambda}' = \frac{\ker \gamma_{\lambda}}{\im \beta_{\lambda}} \oplus \im \gamma_{\lambda} \oplus \frac{\ker \alpha_{\lambda}}{\im \gamma_{\lambda}}
\end{displaymath}
we can consider $M_1'$ to be a graded vector space as
\begin{displaymath}
M_1' \cong \bigoplus_{\lambda \in \mathbb{Z}^4} M_{1,\lambda}'.
\end{displaymath}
Obtaining the desired property for $e^*, w^*, s^*, n^*$ relative to this grading requires making a good choice for the splitting data.  Specifically, we can always arrange to choose $\rho$ and $\sigma$ so that they decompose into blocks
\begin{align*}
\rho_{\lambda} &: M_{1+a,\lambda-(0,1,0,0)} \oplus M_{1+b,\lambda-(1,0,0,0)} \to \ker \gamma_{\lambda} \\
\sigma_{\lambda}&: \ker \alpha_{\lambda} / \im \gamma_{\lambda} \to \ker \alpha_{\lambda}
\end{align*}

By \eqref{eqalphaStar} the nontrivial part of $\left[\begin{array}{cc} e^* & w^*\end{array} \right]$ is given by a pair of maps
\begin{align*}
-\pi \circ \rho &: M_{1+a} \oplus M_{1+b} \to \ker \gamma / \im \beta \\
-\gamma &: M_{1+a} \oplus M_{1+b} \to \im \gamma
\end{align*}
These can be broken down into
\begin{align}
-\pi_{\lambda} \circ \rho_{\lambda} &: M_{1+a,\lambda-(0,1,0,0)} \oplus M_{1+b,\lambda-(1,0,0,0)} \to \ker \gamma_{\lambda} / \im \beta_{\lambda} \subseteq M_{1,\lambda}' \label{eqewStar1}\\
-\gamma_{\lambda} &: M_{1+a,\lambda-(0,1,0,0)} \oplus M_{1+b,\lambda-(1,0,0,0)} \to \im \gamma_{\lambda} \subseteq M_{1,\lambda}' \label{eqewStar2}
\end{align}
where $\pi_{\lambda}: \ker \gamma_{\lambda} \to \ker \gamma_{\lambda} / \im \beta_{\lambda}$ is projection.  Therefore $e^*$ and $w^*$ act on degree by adding $(0,1,0,0)$ and $(1,0,0,0)$ respectively.  A similar argument shows $s^*$ and $n^*$ act by adding $(0,0,0,1)$ and $(0,0,1,0)$ respectively.
\end{proof}

\begin{proof}[Proof of Proposition \ref{proptauGraded}]
Let $(Q,W,M)$ and $(Q',W',M')$ be as before.  Let $M'' = \Theta(M)$.  In short, we are focusing on the representations at frames (1), (2), and (5) of Figure \ref{figtau}.  Steps (3) and (4) do not change any of the vector spaces, and step (5) merely permutes them, so 
\begin{displaymath}
M_v'' = M_{v+1}' = M_{v+1}
\end{displaymath}
for $v=1,\ldots, N-1$ and $M_N'' = M_1'$.  The arrows of $M''$ all correspond to arrows of $M'$, differing by at most a sign which does not affect grading.  By Lemma \ref{lemNewGrading}, each of these arrows acts on graded pieces with a constant shift of degree.  To show that $M''$ is a graded representation of the Jacobian algebra $A$, it remains to show that the degree shift for each arrow agrees with its weight in $\mathbb{Z}^4$.  

Recall that the weight of an arrow is determined by its cardinal direction as per Table \ref{tabArrowTypes}.  As $e^*$ points West (see Figure \ref{figtau}) we have $\wt(e^*) = (0,1,0,0)$ which is in accordance with the result of Lemma \ref{lemNewGrading}.  The weights for $w^*$, $s^*$, and $n^*$ also all match up.  Next consider any composite arrow created in step (2) that survives to step (5), for concreteness let's say $[ws]$.  Recall $[ws]$ acts on $M'$ as the path $ws$ acts on $M$, which is to say
\begin{displaymath}
[ws](M'_{1+c,\lambda}) = w(s(M_{1+c,\lambda})) \subseteq M_{1+b,\lambda+(0,1,1,0)} = M'_{1+b,\lambda+(0,1,1,0)},
\end{displaymath}
so
\begin{displaymath}
[ws](M''_{c,\lambda}) \subseteq M''_{b,\lambda+(0,1,1,0)} = M''_{b,\lambda + \wt([ws])}.
\end{displaymath}
Lastly, every arrow of $Q$ the is untouched through the mutation process winds up corresponding to a different arrow of $Q$ at the end (with $1$ subtracted from each endpoint).  However, the compass direction and hence the weight stay the same as does the associated linear map.

What remains are the statements about degree sets.  Let $S$ be the degree set of $M$, and to better match the notation of this proof, let $S''$ be the degree set of $\Theta(M)$.  Since $M''_v = M_{v+1}$ for $v \neq 1$ we have $S''_v = S_{v+1}$ in this case.  Now consider $M''_N = M'_1$.  Its graded piece $M'_{1,\lambda}$ is built from the maps $\alpha_{\lambda}, \beta_{\lambda}, \gamma_{\lambda}$.  For the piece to be nonzero, one out of the six spaces
\begin{displaymath}
M_{1,\lambda-(1,1,0,0)}, M_{1+a,\lambda-(0,1,0,0)}, M_{1+b,\lambda-(1,0,0,0)}, M_{1+c,\lambda+(0,0,0,1)}, M_{1+d,\lambda+(0,0,1,0)} ,M_{1,\lambda+(0,0,1,1)}
\end{displaymath}
involved in these maps must be nonzero.  To sum up
\begin{align*}
S_{N}'' \subseteq &(S_1 + (1,1,0,0)) \cup (S_{1+a} + (0,1,0,0)) \cup (S_{1+b} + (1,0,0,0) \\
&\cup (S_{1+c} - (0,0,0,1)) \cup (S_{1+d} - (0,0,1,0)) \cup (S_1 - (0,0,1,1))
\end{align*}
The result concerning the value of $a\lambda_1 + b\lambda_2 - c\lambda_3 - d\lambda_4$ follows easily.
\end{proof}

\subsection{The calibrated case} Let $M = M(S,t)$ be a calibrated representation.  By Proposition \ref{proptauGraded}, $\Theta(M)$ is graded so it makes sense to ask if it is calibrated.  We identify the circumstances under which this occurs, beginning with the question of when $\Theta(M)$ is multiplicity free.

\begin{prop} \label{proptauMultFree}
Suppose $S \subseteq X_t$ is finite, connected, interval-closed, and sturdy and $M = M(S,t)$.  Then $\Theta(M)$ is multiplicity free.
\end{prop}

\begin{proof}
Let $M' = \Theta(M)$ and let $S' = S_1' \cup \ldots \cup S_N'$ be its degree set.  Let $\lambda \in S_v'$.  If $v < N$ then $M'_v = M_{v+1}$, so $\dim M'_{v,\lambda} = \dim M_{v+1,\lambda} = 1$.  So the only possibility for higher multiplicity occurs if $\lambda \in S'_N$.  In this case
\begin{displaymath}
M_{N,\lambda}' = \frac{\ker \gamma_{\lambda}}{\im \beta_{\lambda}} \oplus \im \gamma_{\lambda} \oplus \frac{\ker \alpha_{\lambda}}{\im \gamma_{\lambda}}
\end{displaymath}
relative to maps
\begin{displaymath}
A \stackrel{\beta_{\lambda}}{\longrightarrow} B \stackrel{\gamma_{\lambda}}{\longrightarrow} C \stackrel{\alpha_{\lambda}}{\longrightarrow} D
\end{displaymath}
where $A = M_{1,\lambda-(1,1,0,0)}$, $B=M_{1+a,\lambda-(0,1,0,0)} \oplus M_{1+b,\lambda-(1,0,0,0)}$, $C=M_{1+c,\lambda+(0,0,0,1)} \oplus M_{1+d,\lambda+(0,0,1,0)}$, and $D = M_{1,\lambda+(0,0,1,1)}$.  As $M$ is multiplicity free, $A$ and $D$ have dimension at most $1$ while $B$ and $C$ have dimension at most $2$.  One can check cases by case that
\begin{itemize}
\item $\alpha_{\lambda}$ and $\beta_{\lambda}$ always have full rank and
\item $\gamma_{\lambda}$ has full rank except if $\dim B = \dim C = 2$ in which case it has rank $1$.
\end{itemize}
By assumption $\lambda \in S'_N$ so $\dim M'_{N,\lambda} \geq 1$.  We consider the possible ways that could occur and show in each case that $\dim M'_{N,\lambda} = 1$.

First suppose $\ker \gamma_{\lambda} / \im \beta_{\lambda} \neq 0$.  It must be that $\dim B \geq 1$.  If $\dim B = 1$ then $\ker \gamma_{\lambda} = B$ and $\im \beta_{\lambda} = 0$.  It follows that $\dim A = \dim C = 0$.  On the other hand, suppose $\dim B = 2$.  This means that $\lambda - (0,1,0,0)$ and $\lambda -(1,0,0,0)$ are in $S$, so by the sturdy condition $\lambda - (1,1,0,0) \in S$.  Hence $\dim A = 1$ and $\dim \im \beta_{\lambda} = 1$.  Therefore $\dim \ker \gamma_{\lambda} = 2$ and again $C = 0$.  In both cases we have $C = 0$ so that $\im \gamma_{\lambda} = \ker \alpha_{\lambda} / \im \gamma_{\lambda} = 0$.  So $M'_{N,\lambda}$ equals just $\ker \gamma_{\lambda} / \im \beta_{\lambda}$ and has dimension $1$.

Next suppose $\im \gamma_{\lambda}$ is nonzero.  Then $B$ and $C$ must both be non-empty.  In all cases $\gamma_{\lambda}$ has rank $1$ so $\dim \im \gamma_{\lambda} = 1$.  The previous paragraph shows that $\ker \gamma_{\lambda} / \im \beta_{\lambda}$ is empty.  It remains to show $\ker \alpha_{\lambda} / \im \gamma_{\lambda} = 0$, i.e. that $\dim \ker \alpha_{\lambda} = 1$.  If $\dim C = 2$ then $\lambda + (0,0,0,1)$ and $\lambda + (0,0,1,0)$ are both in $S$, so by the sturdy condition $\lambda + (0,0,1,1) \in S$.  Hence $\dim D = 1$.  On the other hand, if $\dim C = 1$ then one out of $\lambda + (0,0,0,1)$ and $\lambda + (0,0,1,0)$ is not in $S$.  Each of these is larger in the partial order than both $\lambda - (0,1,0,0)
$ and $\lambda - (1,0,0,0)$, and one of the latter is in $S$ because $\dim B \geq 1$.  By the interval-closed condition, $S$ cannot contain $\lambda + (0,0,1,1)$, so $\dim D = 0$.  In both cases $\dim D = \dim C - 1$, and since $\alpha_{\lambda}$ is full rank we get $\dim \ker \alpha_{\lambda} = 1$.

Finally suppose $\ker \alpha_{\lambda} / \im \gamma_{\lambda}$ is nonzero.  By the above $\ker \gamma_{\lambda} / \im \beta_{\lambda} = \im \gamma_{\lambda} = 0$.  If $\dim C = 1$ then $\dim D = 0$.  If $\dim C = 2$, then as explained above the sturdy condition forces $\dim D = 1$.  Therefore $\dim \ker \alpha_{\lambda} / \im \gamma_{\lambda} = 1$.
\end{proof}

\begin{table}
\begin{tabular}{l|llllllll}
$\dim A$ & 0 & 1 & 0 & 0 & 1 & 1 & 0 & 0 \\
\hline
$\dim B$ & 1 & 2 & 1 & 1 & 2 & 2 & 0 & 0 \\
\hline
$\dim C$ & 0 & 0 & 1 & 2 & 1 & 2 & 1 & 2 \\
\hline
$\dim D$ & 0 & 0 & 0 & 1 & 0 & 1 & 0 & 1 \\
\hline
$\dim \frac{\ker \gamma_{\lambda}}{\im \beta_{\lambda}}$ & 1 & 1 & 0 & 0 & 0 & 0 & 0 & 0 \\
\hline
$\dim \im \gamma_{\lambda}$ & 0 & 0 & 1 & 1 & 1 & 1 & 0 & 0 \\
\hline
$\dim \frac{\ker \alpha_{\lambda}}{\im \gamma_{\lambda}}$ & 0 & 0 & 0 & 0 & 0 & 0 & 1 & 1 \\
\end{tabular}
\caption{The possible scenarios in which $\dim(M'_{N,\lambda}) = 1$.  The definitions of the spaces $A$, $B$, $C$, and $D$, are given in the proof of Proposition \ref{proptauMultFree}.}
\label{tabSPrime}
\end{table}

We are now ready to prove Theorem \ref{thmMain}, which states that for $S$ as above, 
\begin{equation} \label{eqMain}
\Theta(M(S,t)) = M(S',t+1)
\end{equation}
for some $S'= S_2 \cup S_3 \cup \ldots \cup S_N \cup S'_N$.  In fact, the proof of Proposition \ref{proptauMultFree} indicates how to construct $S_N'$.  Fix $\lambda \in \mathbb{Z}^4$ with $a\lambda_1 + b\lambda_2 - c\lambda_3 - d\lambda_4 = t+N+1$.  Then whether or not $\lambda \in S_N'$ is determined by the dimensions of $A$, $B$, $C$, and $D$, which in turn depend on $S$ (e.g. $A$ has dimension $1$ if $\lambda - (1,1,0,0) \in S$ and $0$ otherwise).  Table \ref{tabSPrime} lists the eight cases under which $\lambda \in S_N'$.  

\begin{proof}[Proof of Theorem \ref{thmMain}]
Suppose $S \subseteq X_t$ is finite, connected, interval-closed, and sturdy.  Let $M' = \Theta(M(S,t))$.  By Proposition \ref{proptauMultFree}, $M'$ is multiplicity free.  By Proposition \ref{proptauGraded}, its degree set $S'$ has
\begin{displaymath}
S'_v \subseteq \{\lambda : a\lambda_1 + b\lambda_2 - c \lambda_3 - d\lambda_4 = t+1+v\}
\end{displaymath}
for $v=1,\ldots, N$.  It remains to check condition (2) in Definition \ref{defCalibrated}.  Since $M=M(S,t)$ is sturdy, it has a basis $\{f_{\lambda} : \lambda \in S\}$ satisfying \eqref{eqCalibrated}.  If $\lambda \in S_v' = S_{v+1}$ for $v=1,\ldots, N-1$ then $M'_{v,\lambda} = M_{v+1,\lambda}$ so we can still use $f_{\lambda}$ as a generator.  We also need a generator of each $M'_{N,\lambda}$ with $\lambda \in S'_N$.  Calling these $f_{\lambda}$ as well, we will show that they can be chosen in a consistent manner.  

Note that in \eqref{eqCalibrated} it suffices to verify the case where the path $\pi$ is a single arrow.  If the arrow is outside the region changed by mutation then the condition clearly survives.  If it is one of the new diagonal arrows, e.g. $[ws]$ (see the bottom of Figure \ref{figtau}) then by definition $[ws]f_{\lambda} = w(s(f_{\lambda}))$.  This equals $f_{\lambda+\wt(s)+\wt(w)}$ (assuming this weight is in $S_{1+b} =S'_b$) by the calibrated property for $M$.  The weight of $[ws]$ in $Q'$ is $(0,1,1,0)$ which equals the sum $\wt(s)+\wt(w)$ of the weights of the arrows in $Q$.  The other new diagonal arrows work similarly.

Finally consider an arrow of $Q'$ incident to $N$, namely $e^*$, $-w^*$, $-s^*$, or $n^*$.  Fix $\lambda \in S'_N$ and restrict to the appropriate graded piece of the arrow, 
\begin{align*}
e^*_{\lambda} &: M'_{a,\lambda-(0,1,0,0)} \to M'_{N,\lambda} \\
-w^*_{\lambda} &: M'_{b,\lambda-(1,0,0,0)} \to M'_{N,\lambda} \\
-s^*_{\lambda} &: M'_{N,\lambda} \to M'_{c,\lambda+(0,0,0,1)} \\
n^*_{\lambda} &: M'_{N,\lambda} \to M'_{d,\lambda+(0,0,1,0)} 
\end{align*}
If the space other than $M'_{N,\lambda}$ in the map is nonzero we need to show the map is an isomorphism.  Whenever this occurs a choice of $f_{\lambda}$ is forced.  For instance if $\lambda-(0,1,0,0) \in S'_a$ then it must be $f_{\lambda} = e^*(f_{\lambda - (0,1,0,0)})$.  So we also must show these choices agree in the case that several of the four maps are nonzero.  There are eight cases, one for each column of Table \ref{tabSPrime}.

In the first case $\dim B = 1$ where 
\begin{displaymath}
B = M'_{a,\lambda-(0,1,0,0)} \oplus M'_{b,\lambda-(1,0,0,0)}
\end{displaymath}
and $M'_{N,\lambda} = \ker \gamma_{\lambda} / \im \beta_{\lambda}$.  By \eqref{eqewStar1}
\begin{equation} \label{eqewStar1Again}
\left[\begin{array}{cc} e_{\lambda}^* & w_{\lambda}^*\end{array} \right] = -\pi_{\lambda} \rho_{\lambda}.
\end{equation}
Since $\dim C = 0$, $\gamma_{\lambda}$ is the zero map which implies $\rho_{\lambda}$ is the identity.  Meanwhile $\pi_{\lambda}$ is always surjective.  One of $e^*_{\lambda}, w^*_{\lambda}$ is forced to be zero since $M'_{a,\lambda-(0,1,0,0)} = 0$ or $M'_{b,\lambda-(1,0,0,0)}=0$.  The other one then is nonzero which is all we need in this case.  The change from $w^*$ to $-w^*$ has no effect.

In the second case $\dim A = 1$ and $\dim B = 2$.  Equation \eqref{eqewStar1Again} still holds and again $\gamma_{\lambda} = 0$ so $\rho_{\lambda}$ is the identity.  By the calibrated property of $M$, $e(f_{\lambda-(1,1,0,0)}) = f_{\lambda-(0,1,0,0)}$ and $w(f_{\lambda-(1,1,0,0)}) = f_{\lambda-(1,0,0,0)}$.  Therefore
\begin{displaymath}
\im \beta_{\lambda} = \im \left[\begin{array}{c} e_{\lambda} \\ w_{\lambda} \end{array} \right]
= \textrm{span} \left[\begin{array}{c} f_{\lambda-(0,1,0,0)} \\ f_{\lambda-(1,0,0,0)} \end{array} \right].
\end{displaymath}
Now $\pi_{\lambda}$ quotients out by $\im \beta_{\lambda}$ so
\begin{displaymath}
e^*(f_{\lambda-(0,1,0,0)}) = \left[\begin{array}{c} -f_{\lambda-(0,1,0,0)} \\ 0 \end{array} \right] + \im \beta_{\lambda}
= \left[\begin{array}{c} 0 \\ f_{\lambda-(1,0,0,0)} \end{array} \right] + \im \beta_{\lambda} = -w^*(f_{\lambda-(1,0,0,0)})
\end{displaymath}
and we get a consistent choice for $f_{\lambda}$.

In the third case $\dim B = \dim C = 1$ and $M'_{N,\lambda} = \im \gamma_{\lambda}$.  By \eqref{eqewStar2}
\begin{displaymath}
\left[\begin{array}{cc} e_{\lambda}^* & w_{\lambda}^*\end{array} \right] = -\gamma_{\lambda}
\end{displaymath}
and by \eqref{eqbetaStar}
\begin{displaymath}
\left[\begin{array}{c} s_{\lambda}^* \\ n_{\lambda}^*\end{array} \right] = \iota_{\lambda}
\end{displaymath}
where $\iota_{\lambda}: \im \gamma_{\lambda} \to M_{1+c,\lambda} \oplus M_{1+d,\lambda}$ is inclusion.  Composing we get
\begin{displaymath}
\left[\begin{array}{cc} s_{\lambda}^*e_{\lambda}^* & s_{\lambda}^*w_{\lambda}^* \\ n_{\lambda}^*e_{\lambda}^* & n_{\lambda}^*w_{\lambda}^* \end{array} \right] = -\gamma_{\lambda}
\end{displaymath}
(the inclusion is built into $\gamma$).  Depending on the makeup of $B$ and $C$, exactly one of $e_{\lambda}^*,w_{\lambda}^*$ is nonzero and exactly one of $s_{\lambda}^*,n_{\lambda}^*$ is nonzero, so one entry of the matrix is possibly nonzero.  On the other hand, by \eqref{eqgamma} each entry equals plus or minus a path in $Q$.  The calibrated property for $M$ gives one of the following four formulas as appropriate
\begin{align*}
s^*e^*(f_{\lambda-(0,1,0,0)}) &= -f_{\lambda+(0,0,0,1)} \\
s^*w^*(f_{\lambda-(1,0,0,0)}) &= f_{\lambda+(0,0,0,1)} \\
n^*e^*(f_{\lambda-(0,1,0,0)}) &= f_{\lambda+(0,0,1,0)} \\
n^*w^*(f_{\lambda-(1,0,0,0)}) &= -f_{\lambda+(0,0,1,0)} 
\end{align*}
A formula for $f_{\lambda}$ always ensues, for instance
\begin{displaymath}
f_{\lambda} = e^*(f_{\lambda-(0,1,0,0)}) = (-s_{\lambda}^*)^{-1}(f_{\lambda+(0,0,0,1)}).
\end{displaymath}

Cases four, five, and six are identical to case three except that at the end more than one of the four formulas apply.  We omit the details for cases seven and eight.  
\end{proof}

In the context of \eqref{eqMain}, write $S' = \Theta(S)$ leaving the dependence on $t$ implicit.  So $\Theta$ is a partial function from finite, connected, interval-closed subsets of $X_t$ to those of $X_{t+1}$ defined for sets that are sturdy.  We will also want to consider $\Theta^{-1}$.  Let $Q^{\opp}$ be the quiver obtained by reversing all arrows of $Q$.  It is easy to check that
\begin{itemize}
\item the vertex map $v \to N+1-v$ induces an isomorphism of $Q$ with $Q^{\opp}$ and
\item $Q^{\opp}$ is a Gale-Robinson quiver with parameters $(a',b',c',d') = (c,d,a,b)$.  
\end{itemize}
As $Q^{\opp}$ is Gale-Robinson, there is a corresponding partial function $\Theta^{\opp}$ from subsets of $X^{\opp}_t$ to subsets of $X^{\opp}_{t+1}$ where
\begin{align*}
X^{\opp}_t &= \{\lambda : t+1 \leq c\lambda_1 + d\lambda_2 - a\lambda_3 - b\lambda_4 \leq t+N\} \\
&= \sigma(X_{-(t+N+1)})
\end{align*}
and $\sigma: \mathbb{Z}^4 \to \mathbb{Z}^4$ interchanges $\lambda_1$ with $\lambda_3$ and $\lambda_2$ with $\lambda_4$.

\begin{prop} \label{propDual}
Let $S \subseteq X_t$.  The following are equivalent:
\begin{itemize}
\item $\Theta^{-1}$ is defined on $S$
\item $\Theta^{\opp}$ is defined on $\sigma(S) \subseteq X^{\opp}_{-(t+N+1)}$
\item $\Theta$ is defined on $-S \subseteq X_{-(t+N+1)}$. 
\end{itemize}
In this case $S' = \Theta^{-1}(S) \subseteq X_{t-1}$ satisfies $\sigma(S') = \Theta^{\opp}(\sigma(S))$ and $-S' = \Theta(-S)$.
\end{prop}

\begin{proof}
Let $M = M(S,t)$.  Let $M^{\opp}$ be the representation of $Q^{\opp}$ obtained by applying the isomorphism $v \mapsto N+1-v$.  It is not hard to see $M^{\opp} = M(\sigma(S), -(t+N+1))$.  Indeed, it is isomorphic to $M$ so it must be calibrated, and e.g. an East arrow $i \to i+a$ of $Q$ has weight $(1,0,0,0)$ while its opposite $i+a \to i$ in $Q^{\opp}$ is a South arrow with weight $(0,0,1,0) = \sigma(1,0,0,0)$.  The isomorphism clearly relates $\mu_N$  with $\mu_1$ and $\rho$ with $\rho^{-1}$.  Therefore $\Theta^{-1} = \rho \circ \mu_N$ on  $Q$ corresponds with $\Theta^{\opp} = \rho^{-1} \circ \mu_1$ on $Q^{\opp}$.  We have proven equivalence of the first two conditions as well as the identity 
\begin{displaymath}
\sigma(\Theta^{-1}(S)) = \Theta^{\opp}(\sigma(S)).
\end{displaymath}

Now let $M' = M(-S,-(t+N+1))$.  We relate the action of $\Theta$ on $M'$ with the action of $\Theta^{\opp}$ on $M^{\opp}$.  Note $-S$ and $\sigma(S)$ are related by the involution 
\begin{displaymath}
\lambda \mapsto \sigma(-\lambda) = (-\lambda_3,-\lambda_4, -\lambda_1, -\lambda_2).
\end{displaymath}  
This map preserves sturdiness of sets, so the second and third conditions above are in fact equivalent.  Suppose $-S$ and $\sigma(S)$ are both sturdy.  We need to show $\Theta(-S)$ and $\Theta^{\opp}(\sigma(S))$ are related by the aforementioned involution.  The only difficulty lies over the vertex $N$, and in each case the set of weights $\lambda$ are those corresponding to the cases listed in Table \ref{tabSPrime}.  The involution interchanges $A$ with $D$ and $B$ with $C$.  By inspection the set of cases in the table are closed under this operation.
\end{proof}

\section{From calibrated representations to cluster variables} \label{secFPolynomials}
Fix a Gale-Robinson quiver $Q$ together with its potential $W = W(Q)$.  Consider the cluster algebra with initial seed $(Q,(x_1,\ldots, x_N))$ and let $z$ be a non-initial cluster variable.  As explained in Section \ref{secDWZ} there is a representation $M = M(z)$ of $(Q,W)$ that encapsulates all the information of $z$.  In the current Section, we focus on the case when $M$ is calibrated, say $M = M(S,t)$.  We explain how to read off information about $z$ from $S$ and then give several examples.

Let $S \subseteq X_t$ be finite, connected and closed under intervals.  An \emph{order filter} of $S$ is a subset $R \subseteq S$ such that $\lambda \in R$, $\mu \in S$, and $\lambda \leq \mu$ imply $\mu \in R$.  Let $\mathcal{F}(S)$ denote the set of order filters of $S$.

\begin{prop}
The graded subrepresentations of $M = M(S,t)$ are precisely the $M(R,t)$ for $R$ an order filter of $S$.
\end{prop}
\begin{proof}
Let $P \subseteq M$ be a graded subrepresentation.  Then $M = \oplus_{\lambda} M_{\lambda}$ and $P = \oplus_{\lambda} P_{\lambda}$ with each $P_{\lambda} \subseteq M_{\lambda}$.  Each $M_{\lambda}$ has dimension at most one, so $P_{\lambda}$ equals either $0$ or $M_{\lambda}$.  As such
\begin{displaymath}
P = \bigoplus_{\lambda \in R} M_{\lambda}
\end{displaymath}
for some subset $R \subseteq S$.  Suppose $\lambda \in R$, $\mu \in S$, and $\lambda \leq \mu$.  By Proposition \ref{proplambda2pi} there is a path $\pi$ from $u = a\lambda_1 + b\lambda_2 - c\lambda_3 - d\lambda_4 - t$ to  $v = a\mu_1 + b\mu_2 - c\mu_3 - d\mu_4 - t$ so that $\wt(\pi) + \lambda = \mu$.  By the definition of calibrated, $\pi$ induces an isomorphism from $M_{\lambda}=P_{\lambda}$ to $M_{\mu}$.  As $P$ is a subrepresentation it must be that $P_{\mu} = M_{\mu}$, i.e. $\mu \in R$.  So $R$ is an order filter as desired.
\end{proof}

Combined with \eqref{eqFPolyFix} and Proposition \ref{propTFixed} we get the following result, which helps to explain the prevalence of order ideals (which are simply order filters of the opposite poset) in $F$-polynomial formulas.

\begin{thm} \label{thmFM}
Let $M = M(S,t)$.  Then
\begin{displaymath}
F(M) = \sum_{R \in \mathcal{F}(S)} y_1^{d_1(R)}\cdots y_N^{d_N(R)}
\end{displaymath}
where $d_i(R) = |R_i|$ is the number of $\lambda \in R$ with $a\lambda_1 + b\lambda_2 - c\lambda_3 - d\lambda_4 = t+i$.  Put another way
\begin{displaymath}
F(M) = \sum_{R \in \mathcal{F}(S)} \prod_{\lambda \in R} y_{a\lambda_1 + b\lambda_2 - c\lambda_3 - d\lambda_4-t}.
\end{displaymath}
\end{thm}

\begin{rmk}
In light of \eqref{eqxFg}, a cluster variable can be recovered from its $F$-polynomial and $g$-vector.  If the variable comes from a calibrated representation then Theorem \ref{thmFM} provides the former.  There is a general explicit formula \cite{DWZ2} for the $g$-vector.  It could be interesting to find a direct description of $g$ in terms of $S$.
\end{rmk}

If $z$ is a non-initial cluster variable, then it appears in some seed obtained from the initial seed by
\begin{displaymath}
\mu_{k_m} \circ \ldots \circ \mu_{k_2} \circ \mu_{k_1},
\end{displaymath}
which we refer to as the mutation sequence $k_1,k_2,\ldots, k_m$.  We always assume $z$ is in position $k_m$ of the final cluster as otherwise fewer mutations would have sufficed.  For convenience, we consider the $k_i$ to be taken modulo $N$.

\begin{prop} \label{proptauz}
Suppose $M = M(z)$ where $z$ corresponds to the mutation sequence $k_1,k_2,\ldots, k_m$.  Then $\Theta(M) = M(z')$ where $z'$ corresponds to the mutation sequence $N, k_1-1, k_2-1,\ldots, k_m-1$ and $\Theta^{-1}(M) = M(z'')$ where $z''$ corresponds to $1,k_1+1,\ldots, k_m+1$.
\end{prop}

\begin{proof}
Let $(Q,\x)$ be the initial cluster and $(Q',\x') = \mu_1(Q,\x)$.  Then by \eqref{eqMz} 
\begin{displaymath}
M_{(Q',\x')}(z) = \mu_1(M_{(Q,\x)}(z)) = \mu_1(M).
\end{displaymath}
To obtain $z$ from $(Q',\x')$, you fist mutate at $1$ to reach $(Q,\x)$ and then do the sequence $k_1,\ldots, k_m$ to get $z$.  So $1,k_1,\ldots, k_m$ is the mutation sequence corresponding to the representation $\mu_1(M)$ of $Q'$.  As $\rho^{-1}$ is an isomorphism decreasing each vertex label by $1$, we have that $N,k_1-1,\ldots, k_m-1$ is the mutation sequence corresponding to the representation $\rho^{-1}\mu_1(M) = \Theta(M)$ of $\rho^{-1}(Q')=Q$.  A similar argument works for $\Theta^{-1}(M)$.
\end{proof}

\begin{prop} \label{propMxi}
Take $(Q,(x_1,\ldots, x_N))$ as an initial seed where $Q$ is a Gale-Robinson quiver.  Let the $x_i$ for $i>N$ or $i\leq 0$ be the non-initial cluster variables defined by \eqref{eqGaleRobinson}.  Then $M(x_i)$ is calibrated for each such $i$.  More precisely
\begin{itemize}
\item $M(x_{1-j}) = M(S^{(j)},j-N-1)$ and
\item $M(x_{N+j}) = M(-S^{(j)},-j)$
\end{itemize}
for each $j > 0$ where 
\begin{equation} \label{eqSj}
\begin{split}
S^{(j)} = \{\lambda \in \mathbb{Z}^4:  j-N &\leq a\lambda_1 + b\lambda_2 - c\lambda_3 - d\lambda_4 \leq j-1, \\
 \lambda_1, \lambda_2  &\geq 0, \lambda_3, \lambda_4 \leq 0\}
\end{split}
\end{equation}
\end{prop}

\begin{proof}
The proof is by induction on $j$.  It follows from the general theory that $M(x_0)$ is the (positive) simple representation at $N$ because $x_0$ is obtained from the initial cluster by a single mutation at $N$.  Meanwhile $S^{(1)} = \{(0,0,0,0)\}$ so $M(S^{(0)}, -N)$ is simple, and the nonzero part is over vertex $v = 0+0-0-0-(-N) = N$.

Now suppose $M(x_{1-j}) = M(S^{(j)},j-N-1)$ and consider the cluster variable $x_{-j}$.  Note $x_{1-j}$ is obtained by the mutation sequence $N,N-1,\ldots, N-j+1$ while $x_{-j}$ is obtained from $N,N-1,\ldots, N-j$.  By Proposition \ref{proptauz}
\begin{displaymath}
M(x_{-j}) = \Theta(M(x_{1-j})).
\end{displaymath}
We claim $S^{(j)}$ is sturdy.  Indeed, let $\lambda \in \mathbb{Z}^4$ with $a\lambda_1 + b\lambda_2 - c\lambda_3 - d\lambda_4 = j-N$.  If $\lambda + (1,0,0,0)$ and $\lambda + (0,1,0,0)$ are in $S(j)$ then $\lambda_1, \lambda_2 \geq 0$ and $\lambda_3 , \lambda_4 \leq 0$ so $\lambda \in S(j)$.  The other sturdy condition is argued similarly.  By Theorem \ref{thmMain}, $M(x_{-j})$ is calibrated, say $M(X_{-j}) = M(S',j-N)$.  The goal is to show $S' = S^{(j+1)}$ which is clear except over vertex $N$.  

Fix $\lambda \in \mathbb{Z}^4$ with $a\lambda_1 + b\lambda_2 - c\lambda_3 - d\lambda_4 = j-N+N = j$.  By the proof of Proposition \ref{proptauGraded}, for $\lambda$ to be in $S'$ at least one of
\begin{displaymath}
\lambda-(1,1,0,0), \lambda-(1,0,0,0), \lambda-(0,1,0,0), \lambda + (0,0,1,0), \lambda + (0,0,0,1), \lambda + (0,0,1,1)
\end{displaymath}
must be in $S = S^{(j)}$.  In all cases we conclude $\lambda_1, \lambda_2 \geq 0$ and $\lambda_3, \lambda_4 \leq 0$ so $\lambda \in S^{(j+1)}$.  Conversely, suppose $\lambda \in S^{(j+1)}$.  Note $j \geq 1$ so at least one of $\lambda_1, \ldots, \lambda_4$ is nonzero and at least one of the aforementioned six weights is in $S$.  The possible patterns for which of these are in $S$ line up exactly with the columns of Table \ref{tabSPrime}, so we get $\lambda \in S'$.

Now consider the variables $x_{N+j}$.  As $x_{N+1}$ results from a single mutation at $1$, $M(x_{N+1})$ is the simple representation at $1$ which equals $M(S^{(1)},-1)$.  For $j \geq 2$, $x_{N+j}$ is reached by the mutation sequence $1,2,\ldots, j$ so $M(x_{N+j}) = \Theta^{-1}(M(x_{N+j-1}))$.  We get by Proposition \ref{propDual} and induction that
\begin{displaymath}
M(x_{N+j}) = M(\Theta^{-(j-1)}(S^{(1)}),-j) = M(-\Theta^{j-1}(-S^{(1)}),-j).
\end{displaymath}
The result follows as $-S^{(1)} = S^{(1)}$ (with $t$ chanced from $-1$ to $-N$) and by the above $\Theta^{j-1}(S^{(1)}) = S^{(j)}$.  
\end{proof}

\begin{thm}
Let $F_i$ be the $F$-polynomial associated to the variable $x_i$ of the Gale-Robinson cluster algebra.  Then for all $j>0$
\begin{equation} \label{eqFNeg}
F_{1-j} = \sum_{R \in \mathcal{F}(S^{(j)})} \prod_{\lambda \in R}y_{a\lambda_1+b\lambda_2-c\lambda_3-d\lambda_4 + (N+1-j)}
\end{equation}
and
\begin{equation} \label{eqFPos}
F_{N+j} = \sum_{I \in J(S^{(j)})} \prod_{\lambda \in I}y_{-a\lambda_1-b\lambda_2+c\lambda_3+d\lambda_4+j}
\end{equation}
where $J(\cdot)$ denotes the set of order ideals of a poset.  In particular, under the specialization $x_1 = x_2 = \ldots = x_N = 1$
\begin{equation} \label{eqxi}
x_{N+j} = |J(S^{(j)})|
\end{equation}
for $j > 0$.
\end{thm}

\begin{rmk}
Formulas for these $F$-polynomials are given in \cite{JMZ} in terms of perfect matchings, and the authors also mention height functions which provide the translation to order ideals.  Another point of comparison is \cite{EF} which discusses order ideal formulas for $F$-polynomials employing the flavored quiver model.  Our contribution is the explicit description \eqref{eqSj} of the posets $S^{(j)}$ which we believe makes \eqref{eqxi} more concise than the other combinatorial formula for the Gale-Robinson sequence in the literature.
\end{rmk}

\begin{proof}
First, \eqref{eqFNeg} follows directly from Theorem \ref{thmFM} and Proposition \ref{propMxi}.  Applying the same approach to the variable $x_{N+j}$ yields
\begin{displaymath}
F_{N+j} = F(M(-S^{(j)}, -j)) 
= \sum_{R \in \mathcal{F}(-S^{(j)})} \prod_{\lambda \in R}y_{a\lambda_1+b\lambda_2-c\lambda_3-d\lambda_4 + j} .
\end{displaymath}
Negation gives a bijection between order filters of $-S^{(j)}$ and order ideals of $S^{(j)}$.  On the level of individual weights, each $\lambda$ is replaced by $-\lambda$ which explains the differing signs in the subscript of \eqref{eqFPos}.

Lastly, suppose $x_1= \ldots = x_N = 1$.  Then by \eqref{eqxFg} and \eqref{eqyHat}
\begin{displaymath}
x_{N+j} = F_{N+j}(1,1,\ldots, 1)
\end{displaymath} 
which by \eqref{eqFPos} equals the number of order ideals of $S^{(j)}$.
\end{proof}

\begin{ex}
Consider the Somos-$4$ quiver with $a=1$, $b=3$, and $c=d=2$.  The posets $S^{(j)}$ for $j=1,2,3$ are given in Figure \ref{figSj}.  Each vertex $\lambda$ is labeled according to $-\lambda_1 - 3\lambda_2 + 2\lambda_3 + 2\lambda_4 + j$ to agree with \eqref{eqFPos}.  The $F$-polynomials are generating functions over order ideals
\begin{align*}
F_5 &= 1+y_1 \\
F_6 &= 1 + y_2 + y_1y_2 \\
F_7 &= 1 + 2y_1 + y_1^2 + y_1^2y_3 + y_1^2y_2y_3 + y_1^3y_2y_3
\end{align*}
The polynomials $F_0$, $F_{-1}$, and $F_{-2}$ are obtained from the same posets by subtracting all vertex weights from $N+1=5$ and taking order filters instead of order ideals.
\end{ex}

\begin{figure}
\begin{pspicture}(-1,0)(12,4.5)
\rput(0,2.5){
\rput(1,1){\rnode{u1}{(0,0,0,0)}}
\rput(-1,1){\rnode{u1L}{$1$}}
}

\rput(1,1){\rnode{v1}{(0,0,0,0)}}
\rput(1,2){\rnode{v2}{(1,0,0,0)}}
\ncline{v1}{v2}
\psset{nodesep=2pt}
\rput(-1,1){\rnode{v1L}{$2$}}
\rput(-1,2){\rnode{v2L}{$1$}}
\ncline{v1L}{v2L}
\psset{nodesep=0pt}

\rput(4,0){
\rput(1,1){\rnode{w1}{(0,0,-1,0)}}
\rput(3,1){\rnode{w2}{(0,0,0,-1)}}
\rput(2,2){\rnode{w3}{(0,0,0,0)}}
\rput(2,3){\rnode{w4}{(1,0,0,0)}}
\rput(2,4){\rnode{w5}{(2,0,0,0)}}
\ncline{w1}{w3}
\ncline{w2}{w3}
\ncline{w3}{w4}
\ncline{w4}{w5}
\rput(5,1){\rnode{w1L}{$1$}}
\rput(7,1){\rnode{w2L}{$1$}}
\rput(6,2){\rnode{w3L}{$3$}}
\rput(6,3){\rnode{w4L}{$2$}}
\rput(6,4){\rnode{w5L}{$1$}}
\psset{nodesep=2pt}
\ncline{w1L}{w3L}
\ncline{w2L}{w3L}
\ncline{w3L}{w4L}
\ncline{w4L}{w5L}
\psset{nodesep=0pt}
}
\psline(3,.5)(3,4.5)
\psline(-1.5,2.75)(3,2.75)
\end{pspicture}
\caption{The posets $S^{(1)}$ (top left), $S^{(2)}$ (bottom left), and $S^{(3)}$ (right)}
\label{figSj}
\end{figure}
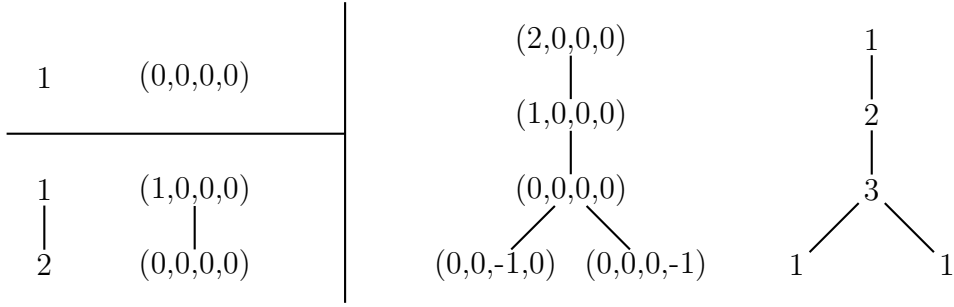

We conclude with a brief discussion of cluster variables other than the $x_i$.  As we will see, only some such variables come from a calibrated representation, but these are a natural family to consider first.  The goal is to classify calibrated representations $M(S,t)$ that correspond to cluster variables.  We have a conjecture along these lines in the special case that $M(S,t)$ is a cyclic module, i.e. $S$ has a unique minimal element.

\begin{conj} \label{conjCyclic}
Fix a vertex $v$ and a set of vertices $\bar{V} \subseteq \{1,2,\ldots, N\}$ that includes $v$.  Choose some $\mu \in X_t$ with $a\mu_1 + b\mu_2 - c\mu_3 -d\mu_4 = t+v$.  Let $S$ be the set of $\lambda \geq \mu$ in $X_t$ with the property that 
\begin{displaymath}
\mu \leq \lambda' \leq \lambda \Longrightarrow a\lambda_1' + b\lambda_2' - c\lambda_3' -d\lambda_4' - t \in \bar{V}.
\end{displaymath}
If $S$ is finite then there is a non-initial cluster variable $z$ so that $M(S,t) = M(z)$. 
\end{conj}

Informally, $M(S,t)$ is obtained by restricting to the vertex set $\bar{V}$ and taking a projective module of the resulting quiver with relations.  Combining the ideas of this paper, we get a large family of combinatorial formulas for $F$-polynomials that (conjecturally) correspond to cluster variables as follows:
\begin{enumerate}
\item Pick $v$ and $\bar{V}$ satisfying the conditions of Conjecture \ref{conjCyclic} in order to obtain $S$.
\item Pick some integer $k$ so that $\Theta^k$ is defined on $S$ to obtain $S'=\Theta^k(S)$.
\item Calculate the generating function $F$ over order filters of $S'$.
\end{enumerate}

\begin{figure}
\begin{pspicture}(0,.5)(6,3.5)
	\rput(1,1){\rnode{w1}{$3$}}
	\rput(2,1){\rnode{w2}{$4$}}
	\rput(3,1){\rnode{w3}{$1$}}
	\rput(4,1){\rnode{w4}{$2$}}
	\rput(5,1){\rnode{w5}{$3$}}
	\rput(1,2){\rnode{V1}{$1$}}
	\rput(2,2){\rnode{V2}{$2$}}
	\rput(3,2){\rnode{V3}{$3$}}
	\rput(4,2){\rnode{V4}{$4$}}
	\rput(5,2){\rnode{V5}{$1$}}
	\rput(1,3){\rnode{W1}{$3$}}
	\rput(2,3){\rnode{W2}{$4$}}
	\rput(3,3){\rnode{W3}{$1$}}
	\rput(4,3){\rnode{W4}{$2$}}
	\rput(5,3){\rnode{W5}{$3$}}
	\ncline{->}{w1}{w2}
	\ncline{->}{w3}{w2}
	\ncline{->}{w3}{w4}
	\ncline{->}{w4}{w5}
	\ncline{->}{w1}{V1}
	\ncline{<-}{w1}{V2}
	\ncline{->}{w2}{V2}
	\ncline{<-}{w3}{V3}
	\ncline{->}{w4}{V3}
	\ncline{<-}{w4}{V4}
	\ncline{->}{w5}{V5}
	\ncline{->}{V1}{V2}
	\ncline{->}{V2}{V3}
	\ncline{->}{V3}{V4}
	\ncline{->}{V5}{V4}
	\ncline{->}{W1}{W2}
	\ncline{->}{W3}{W2}
	\ncline{->}{W3}{W4}
	\ncline{->}{W4}{W5}
	\ncline{->}{W1}{V1}
	\ncline{->}{W2}{V2}
	\ncline{<-}{W3}{V3}
	\ncline{<-}{W4}{V4}
	\ncline{->}{W5}{V5}
	\ncline{->}{V2}{W1}
	\ncline{->}{W4}{V3}
\end{pspicture}
\caption{Part of the lift $\tilde{Q}$ of the Somos-$4$ quiver}
\label{figSomos4}
\end{figure}
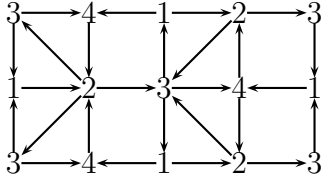

\begin{ex}
Let $Q$ be the Somos-$4$ quiver given in Figure \ref{figSomos4}.  The variables $x_{-1}$, $x_{0}$, $x_5$, $x_6$, and $x_7$ all correspond to cyclic modules.  Indeed, the associated sets $S^{(2)}$, $S^{(1)}$, $-S^{(1)}$, $-S^{(2)}$, $-S^{(3)}$ all have a unique minimal element (see Figure \ref{figSj} and note $-S^{(j)}$ is the opposite poset of $S^{(j)}$).  Each can be expressed in the terms of Conjecture \ref{conjCyclic}.  For instance, $-S^{(3)}$ is obtained by taking $v=1$ and restricting to $\bar{V} = \{1,2,3\}$ starting from the weight $\mu = (-2,0,0,0)$.  
\end{ex}

\begin{ex} \label{exSimple2}
Continuing with the Somos-$4$ quiver, take $v=2$ and $\bar{V} = \{2\}$.  The result $M$ is the simple representation at $2$.  Figure \ref{figSimple2} shows the representations $\Theta^k(M)$ for $k=0,1,\ldots, 4$.  Each is calibrated and the degree set $S$ is encoded by a \emph{contour}, that is a simple closed path in the dual graph of $\tilde{Q}$.  In these cases, 
\begin{itemize}
\item $S$ is in bijection with the set of vertices of $\tilde{Q}$ inside the contour and
\item if there is arrow $\alpha$ of $\tilde{Q}$ from $\lambda \in S$ to $\lambda' \in S$ then $\lambda' = \lambda + \wt(\alpha)$
\end{itemize}
These conditions determine $S$ up to translation.  For instance $\Theta^4(M) = M(S,-1)$ where
\begin{align*}
S_1 &= \{(0,0,0,0)\} \\
S_2 &= \emptyset \\
S_3 &= \{(0,0,-1,0), (0,0,0,-1)\} \\
S_4 &= \{(1,0,-1,0), (1,0,0,-1), (0,1,0,0)\}
\end{align*}
Note $S' = \Theta(S)$ is not sturdy because $S_1' = S_2 = \emptyset$ while $S_3'=S_4$ contains $(1,0,-1,0)$ and $(1,0,0,-1)$.  So $\Theta^5(M)$ is calibrated but $\Theta^6(M)$ is not.

The simple representation $M$ corresponds to the variable obtained by the single mutation at $2$.  By Proposition \ref{proptauz}, the representation $\Theta(M)$ corresponds to $4,1$, $\Theta^2(M)$ corresponds to $4,3,4$, $\Theta^3(M)$ corresponds to $4,3,2,3$ and so on.  Note that in each case the final mutation is not at a degree $4$ vertex.  The first non-calibrated example $\Theta^6(M)$, then, occurs for the mutation sequence $4,3,2,1,4,3,4$.
\end{ex}

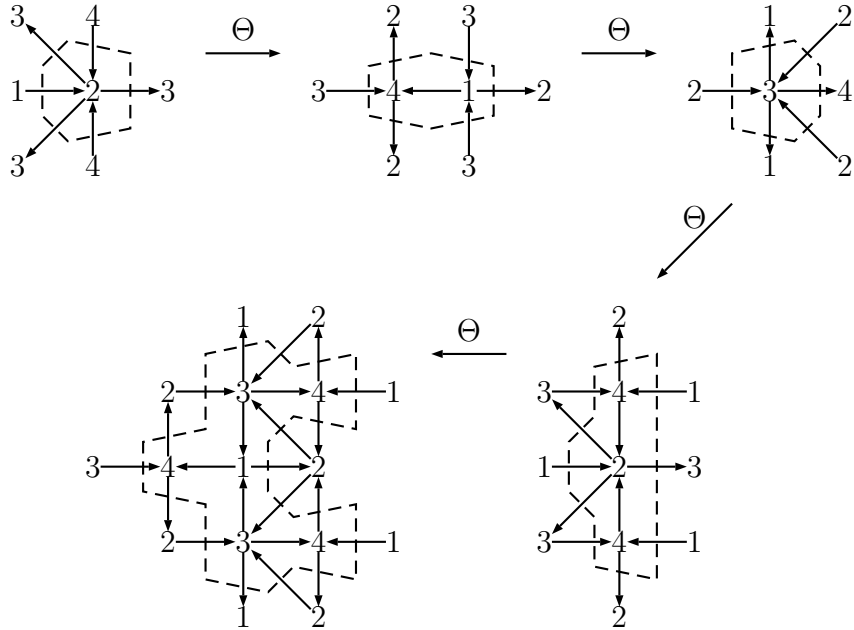
\begin{figure}
\begin{pspicture}(0,0.5)(13,9.5)
\psset{arrowinset=0}
\rput(0,6){
\rput(1,1){\rnode{v11}{$3$}}
\rput(2,1){\rnode{v21}{$4$}}
\rput(1,2){\rnode{v12}{$1$}}
\rput(2,2){\rnode{v22}{$2$}}
\rput(3,2){\rnode{v32}{$3$}}
\rput(1,3){\rnode{v13}{$3$}}
\rput(2,3){\rnode{v23}{$4$}}
\ncline{->}{v22}{v11}
\ncline{->}{v12}{v22}
\ncline{->}{v22}{v13}
\ncline{->}{v23}{v22}
\ncline{->}{v22}{v32}
\ncline{->}{v21}{v22}
\pspolygon[linestyle=dashed](2.5,2.5)(1.67,2.67)(1.33,2.33)(1.33,1.67)(1.67,1.33)(2.5,1.5)
}
\psline{->}(3.5,8.5)(4.5,8.5)
\uput[u](4,8.5){$\Theta$}

\rput(4,6){
\rput(2,1){\rnode{v21}{$2$}}
\rput(3,1){\rnode{v31}{$3$}}
\rput(1,2){\rnode{v12}{$3$}}
\rput(2,2){\rnode{v22}{$4$}}
\rput(3,2){\rnode{v32}{$1$}}
\rput(4,2){\rnode{v42}{$2$}}
\rput(2,3){\rnode{v23}{$2$}}
\rput(3,3){\rnode{v33}{$3$}}
\ncline{->}{v22}{v21}
\ncline{->}{v31}{v32}
\ncline{->}{v12}{v22}
\ncline{->}{v32}{v22}
\ncline{->}{v32}{v42}
\ncline{->}{v22}{v23}
\ncline{->}{v33}{v32}
\pspolygon[linestyle=dashed](3.33,2.33)(2.5,2.5)(1.67,2.33)(1.67,1.67)(2.5,1.5)(3.33,1.67)
}
\psline{->}(8.5,8.5)(9.5,8.5)
\uput[u](9,8.5){$\Theta$}

\rput(9,6){
\rput(2,1){\rnode{v21}{$1$}}
\rput(3,1){\rnode{v31}{$2$}}
\rput(1,2){\rnode{v12}{$2$}}
\rput(2,2){\rnode{v22}{$3$}}
\rput(3,2){\rnode{v32}{$4$}}
\rput(2,3){\rnode{v23}{$1$}}
\rput(3,3){\rnode{v33}{$2$}}
\ncline{->}{v22}{v21}
\ncline{->}{v31}{v22}
\ncline{->}{v12}{v22}
\ncline{->}{v22}{v32}
\ncline{->}{v22}{v23}
\ncline{->}{v33}{v22}
\pspolygon[linestyle=dashed](2.67,2.33)(2.33,2.67)(1.5,2.5)(1.5,1.5)(2.33,1.33)(2.67,1.67)
}
\psline{->}(10.5,6.5)(9.5,5.5)
\uput[u](10,6){$\Theta$}

\rput(7,0){
\rput(2,1){\rnode{v21}{$2$}}
\rput(1,2){\rnode{v12}{$3$}}
\rput(2,2){\rnode{v22}{$4$}}
\rput(3,2){\rnode{v32}{$1$}}
\rput(1,3){\rnode{v13}{$1$}}
\rput(2,3){\rnode{v23}{$2$}}
\rput(3,3){\rnode{v33}{$3$}}
\rput(1,4){\rnode{v14}{$3$}}
\rput(2,4){\rnode{v24}{$4$}}
\rput(3,4){\rnode{v34}{$1$}}
\rput(2,5){\rnode{v25}{$2$}}
\ncline{->}{v22}{v21}
\ncline{->}{v12}{v22}
\ncline{->}{v32}{v22}
\ncline{->}{v23}{v12}
\ncline{->}{v22}{v23}
\ncline{->}{v13}{v23}
\ncline{->}{v23}{v33}
\ncline{->}{v23}{v14}
\ncline{->}{v24}{v23}
\ncline{->}{v14}{v24}
\ncline{->}{v34}{v24}
\ncline{->}{v24}{v25}
\pspolygon[linestyle=dashed](2.5,3.5)(2.5,4.5)(1.67,4.33)(1.67,3.67)(1.33,3.33)(1.33,2.67)(1.67,2.33)(1.67,1.67)(2.5,1.5)(2.5,2.5)
}
\psline{->}(7.5,4.5)(6.5,4.5)
\uput[u](7,4.5){$\Theta$}

\rput(1,0){
\rput(3,1){\rnode{v31}{$1$}}
\rput(4,1){\rnode{v41}{$2$}}
\rput(2,2){\rnode{v22}{$2$}}
\rput(3,2){\rnode{v32}{$3$}}
\rput(4,2){\rnode{v42}{$4$}}
\rput(5,2){\rnode{v52}{$1$}}
\rput(1,3){\rnode{v13}{$3$}}
\rput(2,3){\rnode{v23}{$4$}}
\rput(3,3){\rnode{v33}{$1$}}
\rput(4,3){\rnode{v43}{$2$}}
\rput(2,4){\rnode{v24}{$2$}}
\rput(3,4){\rnode{v34}{$3$}}
\rput(4,4){\rnode{v44}{$4$}}
\rput(5,4){\rnode{v54}{$1$}}
\rput(3,5){\rnode{v35}{$1$}}
\rput(4,5){\rnode{v45}{$2$}}
\ncline{->}{v32}{v31}
\ncline{->}{v41}{v32}
\ncline{->}{v42}{v41}
\ncline{->}{v22}{v32}
\ncline{->}{v32}{v42}
\ncline{->}{v52}{v42}
\ncline{->}{v23}{v22}
\ncline{->}{v32}{v33}
\ncline{->}{v43}{v32}
\ncline{->}{v42}{v43}
\ncline{->}{v13}{v23}
\ncline{->}{v33}{v23}
\ncline{->}{v33}{v43}
\ncline{->}{v23}{v24}
\ncline{->}{v34}{v33}
\ncline{->}{v43}{v34}
\ncline{->}{v44}{v43}
\ncline{->}{v24}{v34}
\ncline{->}{v34}{v44}
\ncline{->}{v54}{v44}
\ncline{->}{v34}{v35}
\ncline{->}{v45}{v34}
\ncline{->}{v44}{v45}
\pspolygon[linestyle=dashed](3.33,3.33)(3.67,3.67)(4.5,3.5)(4.5,4.5)(3.67,4.33)(3.33,4.67)(2.5,4.5)(2.5,3.5)(1.67,3.33)(1.67,2.67)(2.5,2.5)(2.5,1.5)(3.33,1.33)(3.67,1.67)(4.5,1.5)(4.5,2.5)(3.67,2.33)(3.33,2.67)
}
\end{pspicture}

\caption{Part of the $\Theta$ orbit of the simple representation at $2$}
\label{figSimple2}
\end{figure}

\begin{rmk}
There is a general recipe to go from a contour $C$ to a poset $S$ \cite{P} (the rules in Example \ref{exSimple2} only apply when $C$ does not contain a face of $\tilde{Q}$).  Typically the map from $S$ to the set of vertices inside $C$, which can be expressed as $\lambda \in S \mapsto (\lambda_1-\lambda_2, -\lambda_3+\lambda_4) \in \mathbb{Z}^2$, is many to one.  Roughly speaking, $S$ is the largest subset of $X_t$ with the appropriate image under this map.  Not every possible degree set $S$ comes from a contour in this way, although all examples we have observed that give rise to cluster variables do come from contours.
\end{rmk}

\medskip

\textbf{Acknowledgments.}  We thank Gregg Musiker for in large part introducing us to this subject, sharing his thoughts, and pointing us to several key references.  We also thank Dave Anderson for helpful conversations.

\end{document}